\documentclass[a4paper]{article}
\usepackage{RR}
\RRNo{6492}
\usepackage{hyperref}
\usepackage{amssymb, theorem, epsfig, subfigure}
\usepackage{bbm,amsmath, psfrag}
\usepackage{algorithm, algpseudocode}

\newtheorem{thm}{Theorem}[section]
\newtheorem{Def}[thm]{Definition}
\newtheorem{exam}[thm]{Example}
\newtheorem{lemma}[thm]{Lemma}
\newtheorem{rmk}[thm]{Remark}
\newtheorem{rmks}[thm]{Remarks}

\RRdate{April 2008}
\RRauthor{
 Oliver Sch\"utze$^1$, Carlos A. Coello Coello$^1$, Emilia Tantar$^2$ 
 and El-Ghazali Talbi$^1$\\

\vspace{.2cm}

\noindent $^3$ CINVESTAV-IPN, Computer Science Department\\
\noindent e-mail: \{schuetze,ccoello\}@cs.cinvestav.mx 

\vspace{.2cm}

\noindent $^1$ INRIA Futurs, LIFL, CNRS B\^{a}t M3, Cit\'e Scientifique\\
\noindent e-mail: \{emilia.tantar,talbi\}@lifl.fr

}

\authorhead{Sch\"utze, Coello, Tantar, Talbi}

\RRtitle{Computing a Finite Size Representation of the Set of Approximate Solutions 
          of an MOP}
\RRetitle{Computing a Finite Size Representation of the Set of Approximate Solutions 
          of an MOP\footnote{Parts of this manuscript will be published in the
 Proceedings of the Genetic and Evolutionary Computation Conference (GECCO-2008).}}
\titlehead{Representing the $\epsilon$-Efficient Set of an MOP}

\RRresume{Dans des travaux pr\'{e}c\'{e}dent, nous avons propos\'{e} un 
environnement ("framework") pour l'approximation de l'int\'{e}gralit\'{e}
de l'ensemble des solutions $\epsilon$-efficaces (not\'{e} $E_\epsilon$)
d'un probl\`{e}me d'optimisation multi-objectifs \`{a} l'aide d'une 
recherche
stochastique. Il a \'{e}t\'{e} prouv\'{e} que
suivant certaines hypoth\`{e}ses relatives au processus de 
g\'{e}n\'{e}ration de
nouvelles solutions candidates, un tel algorithme produit une s\'{e}quence
d'archives qui converge asymptotiquement vers $E_\epsilon$, au sens
probabiliste du terme. Le r\'{e}sultat, s'il est satisfaisant pour la
plupart des MOP discrets, ne l'est pas d'un point de vue pratique
pour les probl\`{e}mes continus. Dans ce dernier cas, l'ensemble des 
solutions
approxim\'{e}es forme un objet \`{a} $n$ dimentions, o\`{u} $n$ est la 
dimension
de l'espace des param\`{e}tres. Ceci peut amener \`{a} des probl\`{e}mes
de performances puisqu'en pratique la taille de l'archive est finie.

Dans le travail pr\'{e}sent\'{e}, nous nous concentrons sur l'obtention 
d'approximations
finies et pr\'{e}cises de $E_\epsilon$ qui est mesur\'{e} par la distance
de Hausdorff. Nous proposons et nous \'{e}tudions une nouvelle strat\'{e}gie
d'archivage des points de vue th\'{e}orique et pratique. Pour ce faire, nous
analysons le comportement asymptotique de l'algorithme, en fournissant 
les limites
de qualit\'{e} de l'approximation obtenue, aussi bien que la
cardinalit\'{e} de l'approximation et nous pr\'{e}sentons \'{e}galement 
quelques
r\'{e}sultats num\'{e}riques.
}

\RRmotcle{ optimisation multi-objectif, convergence, solutions
$\epsilon$-efficaces, solutions approwim\'{e}es, algorithmes de recherche
stochastique.}

\RRabstract{Recently, a framework for the approximation of the entire set of 
$\epsilon$-efficient solutions (denote by $E_\epsilon$) of a multi-objective optimization 
problem with stochastic search algorithms has been proposed. It was proven that such an 
algorithm produces -- under mild assumptions on the process to generate new candidate solutions 
--a sequence of archives which  converges to $E_{\epsilon}$ in the limit and 
in the probabilistic sense. The result, though satisfactory for most discrete MOPs, 
is at least from the practical viewpoint not sufficient for continuous models: in this case, 
the set of approximate solutions typically forms an $n$-dimensional object, where $n$ denotes 
the dimension of the parameter space, and thus, it may come to perfomance problems since in
practise one has to cope with a finite archive.\\
Here we focus on obtaining finite and tight approximations of $E_\epsilon$, the latter measured
by the Hausdorff distance. We propose and investigate a novel archiving strategy theoretically 
and empirically. For this, we analyze the convergence behavior of the algorithm, yielding bounds 
on the obtained approximation quality as well as on the cardinality of the resulting approximation,
and present some numerical results.

}
\RRkeyword{multi-objective optimization, convergence,
           $\epsilon$-efficient solutions, approximate solutions,
           stochastic search algorithms.}
\RRprojets{Apics et Op\'era}
\RRtheme{\THCom \THCog \THSym \THNum \THBio} 
\URFuturs 
\begin{document}
\makeRR   

\section{Introduction}
Since the notion of $\epsilon$-efficiency for multi-objective optimization problems
(MOPs) has been introduced more than two decades ago (\cite{loridan:84}), this concept 
has been studied and used by many researchers, e.g. to allow (or tolerate) nearly optimal
solutions (\cite{loridan:84}, \cite{white:86}), to approximate the set of optimal solutions 
(\cite{RF1990a}), or in order to discretize this set (\cite{LTDZ2002b}, 
\cite{slcdt:07}). $\epsilon$-efficient solutions or approximate solutions have 
also been used to tackle a variety of real world problems including portfolio selection 
problems (\cite{white:98}), a location problem (\cite{blanquero:02}), or a minimal cost flow 
problem (\cite{RF1990a}). \\
As an illustrative example where it could make sense from the practical point of
view to consider in addition to the exact solutions also approximate ones we consider a 
plane truss design problem, where the volume of the truss as well as the displacement of 
the joint to a given position have to be minimized (see also Section 6.2). Since the designs 
of this problem -- as basically in all other engineering problems -- have to obey certain 
physical contraints such as in this case the weight and  stability of the structural element,
the objective values of all feasible solutions are located within a relatively tight and a
priori appreciable range. 
Hence, the maximal tolerable loss of a design compared to an `optimal' one with respect to 
the objective values can easily be determined quantitatively and qualitatively by the
decision maker (DM) before the optimization process.  
The resulting set of exact {\em and} approximate (but physically relevant) solutions obtained 
by the optimization algorithm\footnote{Here we assume an idealized algorithm, since in 
practise every solution is an approximate one.} 
leads in general to a larger variety of possibilities to the DM than `just' the set of exact
solutions: this is due to the fact that solutions which are `near' in objective space can 
differ significantly in design space (e.g., when the model contains symmetries, or see 
Section 6.3 for another example). \\
The computation of such approximate solutions has been addressed in several studies. In most 
of them, scalarization methods have been empoyed (e.g., \cite{white:86}, \cite{blanquero:02},
\cite{engau:07}). By their nature, such algorithms can deliver only single solutions by one
single execution. The only work so far which deals with the approximation of the entire set of 
approximate solutions (denote by $E_{\epsilon}$) is \cite{SCT_epseff:07}, where an archiving 
strategy for stochastic search algorithms is proposed for this task. Such a sequence of 
archives obtained by this algorithm provably converges -- under mild assumptions on the 
process to generate new candidate solutions -- to $E_{\epsilon}$ in the limit and in the
probabilistic sense. This result, 
though satisfactory for most discrete MOPs, is at least from the practical viewpoint not 
sufficient for continuous models (i.e., continous objectives defined on a continuous  domain): 
in this case, the set of approximate solutions typically forms an $n$-dimensional object, 
where $n$ denotes the dimension of the parameter space (see below).
Thus, it may come to performance problems since it can easily happen that a given threshold
on the magnitude of the archives is exceeded before a `sufficient' approximation of the
set of interest in terms of diversity and/or convergence is obtained. An analogue statement 
holds for the approximation of the Pareto front, which is `only' $(k-1)$-dimensional for MOPs 
with $k$ objectives, and suitable discretizations have been subject of research since several
years (e.g., \cite{LTDZ2002b}, \cite{knowles_2004_bounded}, \cite{slcdt:07}). \\
The scope of this paper is to develop a framework for finite size representations of 
the set $E_{\epsilon}$ with stochastic search algorithms such as evolutionary multi-objective 
(EMO) algorithms. This will call for the design of a novel archiving strategy to store the
`required' solutions found by the stochastic search process. We will further analyze the
convergence behavior of this method, yielding bounds on the approximation quality as well
as on the cardinality of the resulting approximations. Finally, we 
will demonstrate the practicability of the novel approach by several examples. \\

The remainder of this paper is organized as follows: in Section 2, we state the
required background including the set of interest $P_{Q,\epsilon}$. In Section 3, we propose 
a novel archiving strategy for the approximation of $P_{Q,\epsilon}$ and state a convergence 
result, and give further on an upper bound on the resulting archive sizes in Section 4. In 
Section 5, we present some numerical results, and make finally some conclusions in Section 6.

\section{Background}
In the following we consider continuous multi-objective optimization problems 
{\renewcommand{\theequation}{MOP}
\begin{equation} 
\label{eq:MOP}
\min_{x\in Q}\{F(x)\},
\end{equation}
\addtocounter{equation}{-1}
}
where $Q\subset \mathbbm{R}^n$ and $F$ is defined as the vector of the objective 
functions $ F:Q\to\mathbbm{R}^k,\quad F(x) = (f_1(x),\ldots,f_k(x)),$
and where each $f_i: Q\to\mathbbm{R}$ is continuously differentiable. 
Later we will restrict the search to a compact set $Q$, the reader may think of 
an $n$-dimensional box.

\begin{Def}
\begin{enumerate}
\item[(a)]
\label{kleiner}
Let $v,w\in Q$. Then the vector $v$ is {\em less than} $w$ ($v<_pw$),
if $v_i< w_i$ for all $i\in\{1,\ldots,k\}$.  The relation $\leq_p$ is defined analogously.
\item[(b)]
$y\in\mathbbm{R}^n$ is {\em dominated} by a point $x\in Q$ ($x\prec y$) 
with respect to (\ref{eq:MOP}) if $F(x)\leq_p F(y)$ and $F(x)\neq F(y)$, 
else $y$ is called nondominated by $x$.
\item[(c)] $x\in Q$ is called a {\em Pareto point} if there is no 
$y\in Q$ which dominates $x$.
\end{enumerate}
\end{Def}

The set of all Pareto optimal solutions is called the {\em Pareto set} (denote by
$P_Q$). This set typically -- i.e., under mild regularity assumptions -- forms a
$(k-1)$-dimensional object. The image of the Pareto set is called the {\em Pareto front}.\\
We now define another notion of dominance, which is the basis of the approximation concept 
used in this study.
\begin{Def}
Let $\epsilon = (\epsilon_1,\ldots,\epsilon_k)\in\mathbbm{R}^k_+$ and $x,y\in Q$.
\begin{enumerate}
\item[(a)]
  $x$ is said to $\epsilon$-dominate $y$ ($x\prec_\epsilon y$) with respect to (\ref{eq:MOP}) 
  if $F(x)-\epsilon \leq_p F(y)$ and $F(x)-\epsilon \neq F(y)$.
\item[(b)]
  $x$ is said to $-\epsilon$-dominate $y$ ($x\prec_{-\epsilon} y$) with respect to 
  (\ref{eq:MOP}) 
  if $F(x)+\epsilon \leq_p F(y)$ and $F(x)+\epsilon \neq F(y)$.
\end{enumerate}
\end{Def}
The notion of $-\epsilon$-dominance is of course analogous to the `classical'
$\epsilon$-dominance relation but with a value $\tilde{\epsilon}\in\mathbbm{R}^k_-$. However, 
we highlight it here since it will be used frequently in this work. While the
$\epsilon$-dominance is a weaker concept of dominance, $-\epsilon$-dominance is a stronger 
one.\\
We now define the set which we want to approximate in the sequel.
\begin{Def}
Denote by $P_{Q,\epsilon}$ the set of points in $Q\subset \mathbbm{R}^n$  which are not 
$-\epsilon$-dominated by any other point in $Q$, i.e.
\begin{equation}
  P_{Q,\epsilon} := \{ x\in Q | \not\exists y\in Q: y\prec_{-\epsilon}  x
                    \}.
\end{equation}
\end{Def}
To see that $P_{Q,\epsilon}$ typically forms an $n$-dimensional set let 
$x_0\in P_Q$ (such a point, for instance, always exists if $Q$ is compact). 
That is, there exists no $y\in Q$ such that $y\prec x_0$. Since $F$ is continuous 
and $\epsilon\in\mathbbm{R}^k_+$ there exists a neigborhood $U$ of $x_0$ such that
\begin{equation}
 \not\exists y\in Q:\; y\prec_{-\epsilon} u \quad \forall u\in U\cap Q,
\end{equation} 
and thus, $U\cap Q\subset P_{Q,\epsilon}$, and we are done since $U$ is
$n$-dimensional.\\
The following result and notions are used for the upcoming proof of convergence.

\begin{thm}[\cite{SSW:02}]
\label{thm:schaeffler}
Let (MOP) be given and $q:\mathbbm{R}^n\to\mathbbm{R}^n$ be defined by 
$ q(x) = \sum_{i=1}^k\hat{\alpha}_i\nabla f_i(x)$, where $\hat{\alpha}$ is a solution of 
\begin{equation*} \label{eq:qop}
\min_{\alpha\in\mathbbm{R}^k}
  \left\{\left\|\sum_{i=1}^k\alpha_i\nabla f_i(x)\right\|_2^2; \alpha_i\geq0,i=1,\ldots,k,
   \sum_{i=1}^k\alpha_i = 1\right\}.
\end{equation*}
Then either $q(x)=0$ or $-q(x)$ is a descent direction for all objective 
functions $f_1,\ldots,f_k$ in $x$. Hence, each $x$ with $q(x)=0$ fulfills the
first-order necessary condition for Pareto optimality.
\end{thm}

\begin{Def}
Let $u\in\mathbbm{R}^n$ and $A,B\subset \mathbbm{R}^n$. The semi-distance
$\mbox{dist}(\cdot,\cdot)$ and the {\em Hausdorff distance} $d_H(\cdot,\cdot)$ are
defined as follows:
\begin{itemize}
  \item[(a)] $\mbox{dist}(u,A) := \inf\limits_{v\in A} \|u-v\|$
  \item[(b)] $\mbox{dist}(B,A) := \sup\limits_{u\in B}\,\mbox{dist}(u,A)$
  \item[(c)] $ d_H(A,B) := \max\left\{ \mbox{dist}(A,B), \mbox{dist}(B,A) \right\} $
\end{itemize}
\end{Def}

Denote by $\overline{A}$ the closure of a set $A\in\mathbbm{R}^n$, by 
$\overset{\circ}{A}$ its interior, and by $\partial A = \overline{A}\backslash 
\overset{\circ}{A}$ the boundary of $A$. 

\begin{Def}
\begin{enumerate}
\item[(a)]
A point $x\in Q$ is called a {\em weak Pareto point} if there exists
no point $y\in Q$ such that $F(y)<_p F(x)$.
\item[(b)]
A point $x\in Q$ is called {\em $-\epsilon$ weak Pareto point} if there exists
no point $y\in Q$ such that $F(y)+\epsilon <_p F(x)$.
\end{enumerate}
\end{Def}

Algorithm \ref{alg:generic_emo} gives a framework of a generic stochastic
multi-objective optimization algorithm, which will be considered in this work. Here, 
$Q\subset\mathbbm{R}^n$ denotes the domain of the MOP, $P_j$ the candidate set 
(or population) of the generation process at iteration step $j$, and $A_j$ the 
corresponding archive.
\begin{algorithm}
\caption {Generic Stochastic Search Algorithm}
\label{alg:generic_emo}
\begin{algorithmic}[1]
\State $P_0\subset Q$ drawn at random
\State $A_0 = ArchiveUpdate(P_0,\emptyset)$
\For{$j=0,1,2, \ldots$}
  \State $P_{j+1} = Generate (P_j)$
  \State $A_{j+1} = ArchiveUpdate (P_{j+1}, A_j)$
\EndFor
\end{algorithmic}
\end{algorithm}

\section{The Algorithm}
Here we present and analyze a novel archiving strategy which aims for a finite 
size representation of $P_{Q,\epsilon}$. \\
The algorithm which we propose here, $ArchiveUpdateP_{Q,\epsilon}$, is given in
Algorithm \ref{alg:PQe2}. Denote by $1\Delta := (\Delta,\ldots,\Delta)\in \mathbbm{R}_+^k$, 
where $\Delta\in\mathbbm{R}_+$ can be viewed as the discretization parameter of the
algorithm.

\begin{algorithm}
\caption {$A := ArchiveUpdateP_{Q,\epsilon}\; (P, A_0, \Delta)$}
\label{alg:PQe2}
\begin{algorithmic}[1]
\Require{population $P$, archive $A_0$, $\Delta\in\mathbbm{R}_+$, $\Delta^*\in (0,\Delta)$}
\Ensure{updated archive $A$}
\State $A := A_0$
\ForAll{$p\in P$}
  \If{$\not\exists a_1\in A:\; a \prec_{-\epsilon} p \;\mbox{and}\; 
       \not\exists a_2\in A:\;d_\infty(F(a),F(p))\leq \Delta^*$} 
  \State $A:=A\cup \{p\}$
  \ForAll{$a\in A$} 
    \If{$p \prec_{-(\epsilon+1\Delta)} a$}
       \State $A:=A\backslash\{a\}$
    \EndIf
  \EndFor
 \EndIf
\EndFor 
\end{algorithmic}
\end{algorithm}

\begin{lemma}
\label{lemma:au1}
Let $A_0, P\subset \mathbbm{R}^n$ be finite sets, $\epsilon\in\mathbbm{R}_+^k$, 
$\Delta\in\mathbbm{R}_+$, and \\
$A:= \mathit{ArchiveUpdateEps1} \; (P, A_0)$. Then the following holds:
\[ \forall x\in P\cup A_0:\quad \exists a\in A: a \prec_{1\Delta} x.
\]
\end{lemma}

\begin{proof} This follows immediately by the construction of the algorithm.
\end{proof}

\begin{thm}
\label{thm:convergence}
Let an MOP $F:\mathbbm{R}^n\to\mathbbm{R}^k$ be given, where $F$ is continuous, let
$Q\subset\mathbbm{R}^n$ be a compact set and $\epsilon\in\mathbbm{R}^k_+$, 
$\Delta,\Delta^*\in\mathbbm{R}_+$ with $\Delta^*<\Delta$. For the generation process
we assume
\begin{equation}
\begin{split}
  \label{eq:P=1} 
  \forall x\in Q \;\mbox{and}\; \forall \delta > 0:
  \quad
  P\left(\exists l\in\mathbbm{N}\; : \; P_l\cap B_\delta(x)\cap Q\neq 
   \emptyset\right) = 1
\end{split}
\end{equation}
and for the MOP 
\begin{itemize}
\item[(A1)] Let there be no weak Pareto point in $Q\backslash P_Q$.
\item[(A2)] Let there be no $-\epsilon$ weak Pareto point in 
            $Q\backslash\overline{P_{Q,\epsilon}}$,
\item[(A3)] Define ${\cal B} := \{x\in Q | \exists y\in P_Q: F(y)+\epsilon = F(x)\}$.
            Let ${\cal B} \subset \overset{\circ}{Q}$ and $q(x)\neq 0$ for all $x\in {\cal B}$, 
            where $q$ is as defined in Theorem \ref{thm:schaeffler}.
\end{itemize}

Then, an application of Algorithm \ref{alg:generic_emo}, where \\
$AchiveUpdateP_{Q,\epsilon}(P,A,\Delta)$ is used to update the archive, leads to a sequence 
of archives  $A_l, l\in\mathbbm{N}$, where the following holds:
\begin{itemize}
\item[(a)] For all $l\in\mathbbm{N}$ it holds
 \begin{equation}
  \|F(a_1)-F(a_2) \|_\infty \geq \Delta^*
 \end{equation}

\item[(b)] There exists with probability one an $l_0\in\mathbbm{N}$ such that
  for all $l\ge l_0$:
\begin{itemize}
  \item[(b1)] $dist (F(P_{Q,\epsilon}),F(A_l)) < \Delta$
  \item[(b2)] $dist (F(A_l),F(P_{Q,\epsilon})) \leq 
              dist(F(P_{Q,\epsilon+2\Delta}),F(P_{Q,\epsilon}))$
  \item[(b3)] $d_H (F(P_{Q,\epsilon}),F(A_l)) \leq D$, where\\
              $D=max(\Delta, dist(F(P_{Q,\epsilon+2\Delta}),F(P_{Q,\epsilon}))$
\end{itemize}
\end{itemize}
\end{thm}

\begin{proof}
Before we state the proof we have to make some remarks: a point $p$ is discarded from an 
existing archive $A$ in two cases (see line 3 of Algorithm \ref{alg:PQe2}):
\begin{equation}
\label{eq:d1d2}
\begin{split}
  (D1) & \quad \exists a_1\in A:\; a_1\prec_{-\epsilon} p, \quad \mbox{or} \\
  (D2) & \quad \exists a_2\in A:\; \|F(a_2)-F(p)\|_\infty\leq \Delta^*.
\end{split}
\end{equation}

Further, we define by
\[ B^\infty_\delta(x) := \{y\in\mathbbm{R}^k\, :\, \|x-y\|_\infty<\delta\}
\]
a $k$-dimensional open box around $x\in\mathbbm{R}^k$. Now we are in the position to state 
the proof.

{\it Claim (a)}: follows immediately by construction of the algorithm and by an inductive
argument. \\
{\it Claim (b1)}: By \it{(a)} it follows that for an element $a$ from a given archive $A$ 
it holds
\begin{equation}
\label{eq:exclusion}
F(\tilde{a})\not\in B^\infty_{\Delta^*}(F(a)),\quad \forall \tilde{a}\in A\backslash \{a\},
\end{equation}
Since further $Q$ is compact and $F$ is continuous it follows that $F(Q)$ is bounded,
and thus, there exits an upper bound for the number of entries in the archive for a 
given MOP, denote by $n_0 = n_0 (\Delta^*, F(Q))$ (see also Section 4). \\
Since $\overline{P_{Q,\epsilon}}$ is compact and \\
$dist(F(P_{Q,\epsilon}),F(A_l)) = dist(F(\bar{P_{Q,\epsilon}}),F(A_l))$, 
and since  $A_l, l\in\mathbbm{N}$, is finite it follows that 
\[ dist(F(P_{Q,\epsilon}),F(A_l)) 
   = \max\limits_{p\in P_{Q,\epsilon}} \min\limits_{a\in A_l} \|F(p)-F(a)\|_\infty
\] 
That is, the claim is right for an archive $A_l, l\in\mathbbm{N},$ if for every 
$p\in P_{Q,\epsilon}$ there exists an element $a\in A_l$ such that 
$\|F(p)-F(a)\|_\infty < \Delta$. Thus, $F(P_{Q,\epsilon})$ must be contained in 
$C_{A_l,\Delta}$, where
\[ C_{A,\Delta} := \bigcup_{a\in A} B_\Delta^\infty (F(a)).
\]
First we show that if there exists an $l_0\in\mathbbm{N}$ with\\
$dist(F(P_{Q,\epsilon}),F(A_l)) < \Delta$, this property holds for all $l\geq l_0$. 
Assume that such an $l_0$ is given. Define
\begin{equation}
  \tilde{A} := \left\{ a\in A_{l_0} | \exists p\in P_{Q,\epsilon}: 
              \|F(p)-F(a) \|_\infty < \Delta \right\}
\end{equation}
Since it holds that
\[ p\in P_{Q,\epsilon}\;\mbox{and}\; a\in Q:\|F(p)-F(a)\|\leq \Delta \;\Rightarrow
   a\in P_{Q,\epsilon+1\Delta}
\]
it follows that $\tilde{A}\subset P_{Q,\epsilon+1\Delta}$, and thus, no element 
$a\in \tilde{A}$ will be discarded further on due to the construction of 
$ArchiveUpdateP_{Q,\epsilon}$. Since $dist(F(P_{Q,\epsilon}),F(A_l)) < \Delta$ it 
follows that for all $p\in P_{Q,\epsilon}$ there exists an element $a\in \tilde{A}$
such that $\|F(p)-F(a)\|_\infty<\Delta$. By the above discussion this holds for 
all $l\geq l_0$, and since no element $a\in \tilde{A}$ is discarded during the run
of the algorithm, and the claim follows. \\
It remains to show the existence of such an integer $l_0$, which we will do by 
contradiction: first we show that by using $ArchiveUpdateP_{Q,\epsilon}$ and under 
the assumptions made above only finitely many replacements can be done during the run 
of the algorithm. Then we construct a contradiction by showing that under the asssumptions
made above infinitely many replacements have to be done during the run of the algorithm 
with the given setting.\\
Let a finite archive $A_0$ be given. If a point $p\in\mathbbm{R}^n$ 
replaces a point $a\in A_0$ (see lines 4 and 7 of Algorithm \ref{alg:PQe2}) it 
follows by construction of $ArchiveUpdateP_{Q,\epsilon}$ that 

\begin{equation}
\label{eq:reduction}
  F(p)<_p F(a)-\Delta
\end{equation}

Since the relation `$\prec$' is transitive, there exists for  every $a\in A$ a `history' 
of replaced points 
$a_i\in A_{l_i}$ where equation (\ref{eq:reduction}) holds for $a_i$ and $a_{i-1}$. 
Since $F(Q)$ is bounded there exist $l_i,u_i\in\mathbbm{R}, i=1,\ldots,k$, such that
$F(Q)\subset [l_1,u_1]\times\ldots\times[l_k,u_k]$. After $r$ replacements there
exists at least one $a\in A_{l(r)}$ such that the length $h$ of the history of $a$ is at least
$h \geq \lceil r/n_0 \rceil$, where $n_0$ is the maximal number of entries in the archive
(see above). Denote by $a_0\in A_0$ the root of the history of $a$. For $a,a_0$ it follows that 
\[ F(a)\leq F(a_0)-h\Delta
\]
For 
$\tilde{h} > h_{max}:= \Delta^{-1}\max_{i=1,\ldots,k} u_i-l_i$ we obtain a contradiction since 
in that case there exists $i\in\{1,\ldots,n\}$ with $f_i(a)<l_i$ and thus $F(a)\not\in F(Q)$. 
Hence it follows that there can be done only finitely many such replacements during the run of 
an algorithm.\\

Assume that such an integer $l_0$ as claimed above does not exist, that is, that 
$F(P_{Q,\epsilon})\not\subset C_{A_l,\Delta}$ for all $l\in\mathbbm{N}$. Hence there exists 
a sequence of points
\begin{equation}
\label{eq:hp1}
 p_i\in P_{Q,\epsilon}:\quad y_i=F(p_i)\in F(P_{Q,\epsilon})\backslash C_{A_i,\Delta}\quad \forall i\in\mathbbm{N}.
\end{equation}

Since $P_{Q,\epsilon}\subset Q$ and $Q$ is compact there exists an accumulation point 
$p^*\in \overline{P_{Q,\epsilon}}$, that is, there exists a subsequence $\{i_j\}_{j\in\mathbbm{N}}$ with
\begin{equation}
\label{eq:hp2}
 p_{i_j} \to p^*\in\overline{P_{Q,\epsilon}}\; \mbox{for} \; j\to \infty. 
\end{equation}
In \cite{SCT_epseff:07} it was shown that under the assumptions (A1)--(A3) it follows that
\begin{equation}
  \overline{\overset{\circ}{P_{Q,\epsilon}}} = \overline{P_{Q,\epsilon}},
\end{equation}
i.e., that $P_{Q,\epsilon}$ is not `flat' anywhere. Hence, the set 
\begin{equation}
 \tilde{U}_1:= B_{(\Delta - \tilde{\Delta})/2} ^\infty (y^*) 
  \cap \overset{\circ}{P_{Q,\epsilon}},
\end{equation}
where $y^*:=F(p^*)$, is not empty. By (\ref{eq:P=1}) it follows that there exists 
with probability one an $l_1\in\mathbbm{N}$ and an element
$\tilde{x}_1\in P_{l_0 + l_1}$ generated by {\em Generate()} with $\tilde{y}_1 = F(\tilde{x}_1)\in \tilde{U}_1$.  
There are two cases for the archive $A_{l_0+l_1}$: (a) $x_1$ can be discarded from the archive, or (b) $x_1$ is 
added to it. 
Assume first that $x_1$ is discarded. Since $x_1\in P_{Q,\epsilon}$ there exists no $\bar{x}\in Q$ such that
$\bar{x}$ $-\epsilon$-dominates $x_1$. Hence, $(D1)$ can not occur (see (\ref{eq:d1d2})), and thus, there must exist
an $a_2\in A_{l_0+l_1}$ such that $\|F(a_2)-F(x_1)\|_\infty\leq \Delta^*$ (see $(D2)$). 
Thus, whether $x_1$ is added to the 
archive or not there exists an $\tilde{a}_1\in A_{l_0+l_1}$ such that $\|F(\tilde{a}_1-y^*\|_\infty\leq \Delta$ 
(since in case $x_1$ is added to the archive $\tilde{a}_1=x_1$ can be chosen), and we obtain 
\begin{equation}
\label{eq:distanceDelta}
\|F(\tilde{a}_1)-\tilde{y}\|_\infty \leq \|F(\tilde{a}_1)-F(x_1)\|_\infty + \|F(x_1)-\tilde{y}\|_\infty  < \Delta \quad \forall \tilde{y}\in U_1
\end{equation}

By (\ref{eq:hp1}) and (\ref{eq:hp2}) there exist integers $j_1,\tilde{l}_1\in\mathbbm{N}$ with 
\begin{equation}
   y_{i_{j_1}}\in \tilde{U}_1\backslash C_{l_0+l_1+\tilde{l}_1,\Delta}.
\end{equation}
Since by (\ref{eq:distanceDelta}) it holds that $\|y_{i_{j_1}} - F(a_1)\|_\infty<\Delta$
it follows that $a_1\not\in A_{l_0+l_1+\tilde{l}_1}$, which is only possible via a
replacement in Algorithm \ref{alg:PQe2} (lines 4 and 7). \\
In an analogous way a sequence $\{a_i\}_{i\in\mathbbm{N}}$ of elements can be constructed
which have to be replaced by other elements. Since this leads to a sequence of
infinitely many replacements. This is a contradiction to the assumption, and the
proof is complete.\\
{\it Claim (b2)}: Let $\tilde{A}$ and $l_0$ as above, and let $l\geq l_0$. Further, let 
$x\in Q\backslash P_{Q,\epsilon+2\Delta}$, that is, there exists a $p\in P_{Q,\epsilon}$
such that $p\prec_{-(\epsilon+2\Delta)} x$. Since $l\geq l_0$ there exists an 
$a\in\tilde{A}\subset A_l$ such that $\|F(p)-F(a)\|_\infty <\Delta$. Combining both facts
we see that $a\prec_{-(\epsilon+1\Delta)} x$. Thus, no element 
$x\in Q\backslash P_{Q,\epsilon+2\Delta}$ is contained in $A_l, l\geq l_0$, or will ever
be added to the archive further on. The claim follows since the archive can only contain
elements in $P_{Q,\epsilon+2\Delta}$ (see also Examples \ref{ex:dist1} and \ref{ex:dist2}).\\
{\it Claim (b3)}: follows immediately by (b1) and (b2).
\end{proof}

\begin{rmks}
\label{rmks:delta}
\begin{enumerate}
\item[(a)] For $\Delta = \Delta^* = 0$ the archiver coincides with the one proposed in
\cite{SCT_epseff:07}, which reads as 
\begin{equation}
\label{eq:update}
  UpdateP_{Q,\epsilon}(P,A) :=
     \{x \in P \cup A : y \not\prec_{-\epsilon} x \,\, \forall y \in P \cup A\}.
\end{equation}
\item[(b)]
The convergence result holds for a scalar $\Delta_0\in\mathbbm{R}_+$ which is used for
the discretization of the $\epsilon$-efficient front. However, analogue results can 
be obtained by using a vector $\Delta\in\mathbbm{R}_+^k$. In this case, the exclusion
strategy in line 3 of Algorithm 2 has to be replaced by
\begin{equation}
  \not\exists a_2\in A:\quad F(p) \in B(F(a_2),\Delta),
\end{equation}
where 
\[ B(y,\Delta):= \{x\in\mathbbm{R}^k\, :\, |x_i-y_i|\leq  \Delta_i,\, i=1,..,k \}.
\]
Further, elements $a$ have to be discarded from the archive if they are 
$-(\epsilon+\Delta)$ dominated by $p$ (lines 6-8).

\item[(c)] In the algorithm the discretization is done in the image space (line
3 of Algorithm 2). By replacing this exclusion strategy by
\begin{equation}
 \not\exists a_2\in A: \; d_\infty(a,p)\leq \Delta^*,
\end{equation}
an analogue result with discretization in parameter space can be obtained. 
This will lead on one hand to approximations which could be more 'complete', but 
on the other hand certainly to archives with much larger magnitudes since 
$P_{Q,\epsilon}$ is $n$-dimensional (see also the discussion in Section 4).

\item[(d)] The parameter $\Delta^*\in\mathbbm{R}_+$ with $\Delta^*<\Delta$ is used
for theoretical purposes. In practise, $\Delta^*=\Delta$ can be chosen.

\item[(e)] Note that the convergence result also holds for discrete models. 
In that case, assumption  (\ref{eq:P=1}) can be modified using Markov chains such 
that it can easier be verified (see e.g., \cite{Rudolph00}).
\end{enumerate}
\end{rmks}

The next two examples show that with using 
$ArchiveUpdateP_{Q,\epsilon}$ one cannot
prevent to maintain points $x\in P_{Q,\epsilon + 2\Delta}\backslash P_{Q,\epsilon}$ in
the limit archive, and that the distance between $F(P_{Q,\epsilon + 1\Delta})$ 
(respectively $F(P_{Q,\epsilon + 2\Delta})$) and $F(P_{Q,\epsilon})$ can get large
in some (pathological) examples.

\begin{exam}
\label{ex:dist1}
Consider the following MOP:
\begin{equation}
  F:\mathbbm{R}\to\mathbbm{R},\qquad F(x)=x
\end{equation}
Let $Q=[0,5]$, $\epsilon=1, \Delta=0.1$, and let $\Delta^*=\Delta$. Thus, we have
$P_{Q,\epsilon}=[0,1]$.
Assume that $A=\{a_1\}$ with $a_1=1.2$. If next $a_2=0.1$ is considered, it will be 
inserted into the archive since $d_\infty(F(a_1),F(a_2))> \Delta$ and since 
$a_2\in P_{Q,\epsilon}$ is not $-\epsilon$-dominated by $a_1$ nor by any other point
$x\in Q$, and will thus remain in the archive further on. Since $a_2$ is not
$-(\epsilon+\Delta)$-dominating $a_1$ we have for the updated archive $A=\{a_1, a_2\}$.
Hence, no element $a\in[0,\Delta]$ will be taken to the archive since for these
points it holds $d_\infty(F(a),F(a_2))\leq \Delta^*$, and thus, 
$a_2\in P_{Q,\epsilon + 2\Delta}\backslash P_{Q,\epsilon}$ will not be discarded from
the archive during the run of the algorithm. \\
When on the other side $a_1=0$ is taken to the archive, no element 
$a\in Q\backslash P_{Q,\epsilon}$ will ever be accepted further on.
\end{exam}

\begin{exam}
\label{ex:dist2}
Let the MOP be given by $F:\mathbbm{R}\to\mathbbm{R}^2$, where 
\begin{equation}
\label{eq:exam2}
  f_1(x) = |x+1|,\qquad f_2(x) = \left\{\begin{tabular}{r@{\quad}l} $|x-1|$ & for $x\leq 1$\\
  $\alpha |x-1|$& for $x>1$\end{tabular}\right. ,
\end{equation}
where $\alpha\in (0,1)$ (see also Figure \ref{fig:exam2}). For simplicity we assume that 
$\epsilon=(\bar{\epsilon},\bar{\epsilon})\in\mathbbm{R}_+^2$. It is $P_Q = [-1,1]$
with 
\begin{equation}
F(-1) = (0, 2),\qquad F(1) = (2,0)
\end{equation}
Further, it is
\begin{equation}
F(-1-\bar{\epsilon}) = 
   (\bar{\epsilon}, 2+\bar{\epsilon}),\qquad F(1+\frac{\bar{\epsilon}}{\alpha}) 
               = (2+\frac{\bar{\epsilon}}{\alpha},\bar{\epsilon})
\end{equation}
Using this and some monoticity arguments on $f_1$ and $f_2$ we see that 
\begin{equation}
  P_{Q,\epsilon} = \left(-1-\bar{\epsilon}, 1+\frac{\bar{\epsilon}}{\alpha}\right]
\end{equation}
Since $F(1+\frac{\bar{\epsilon}+\Delta}{\alpha}) 
               = (2+\frac{\bar{\epsilon}+\Delta}{\alpha},\bar{\epsilon}+\Delta)$ 
it follows that 
\begin{equation}
 dist(F(P_{Q,\epsilon + 1\Delta}), F(P_{Q,\epsilon})) = \frac{\Delta}{\alpha},
\end{equation}
which can get large for small values of $\alpha$.

\begin{figure}
\begin{center}
\subfigure[Parameter space]{
\epsfig{file=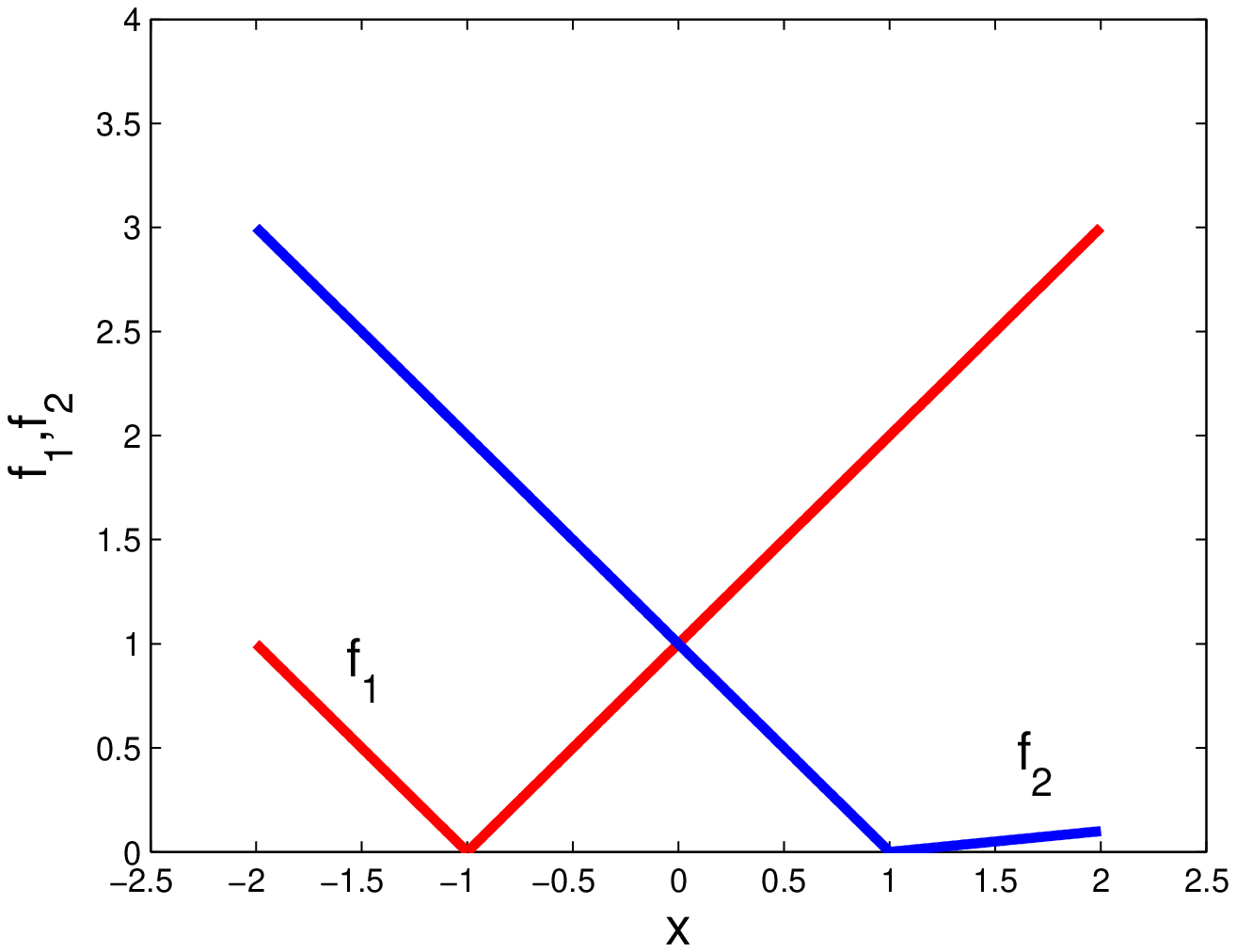,width=6cm}}
\hfill
\subfigure[Image space]{
\epsfig{file=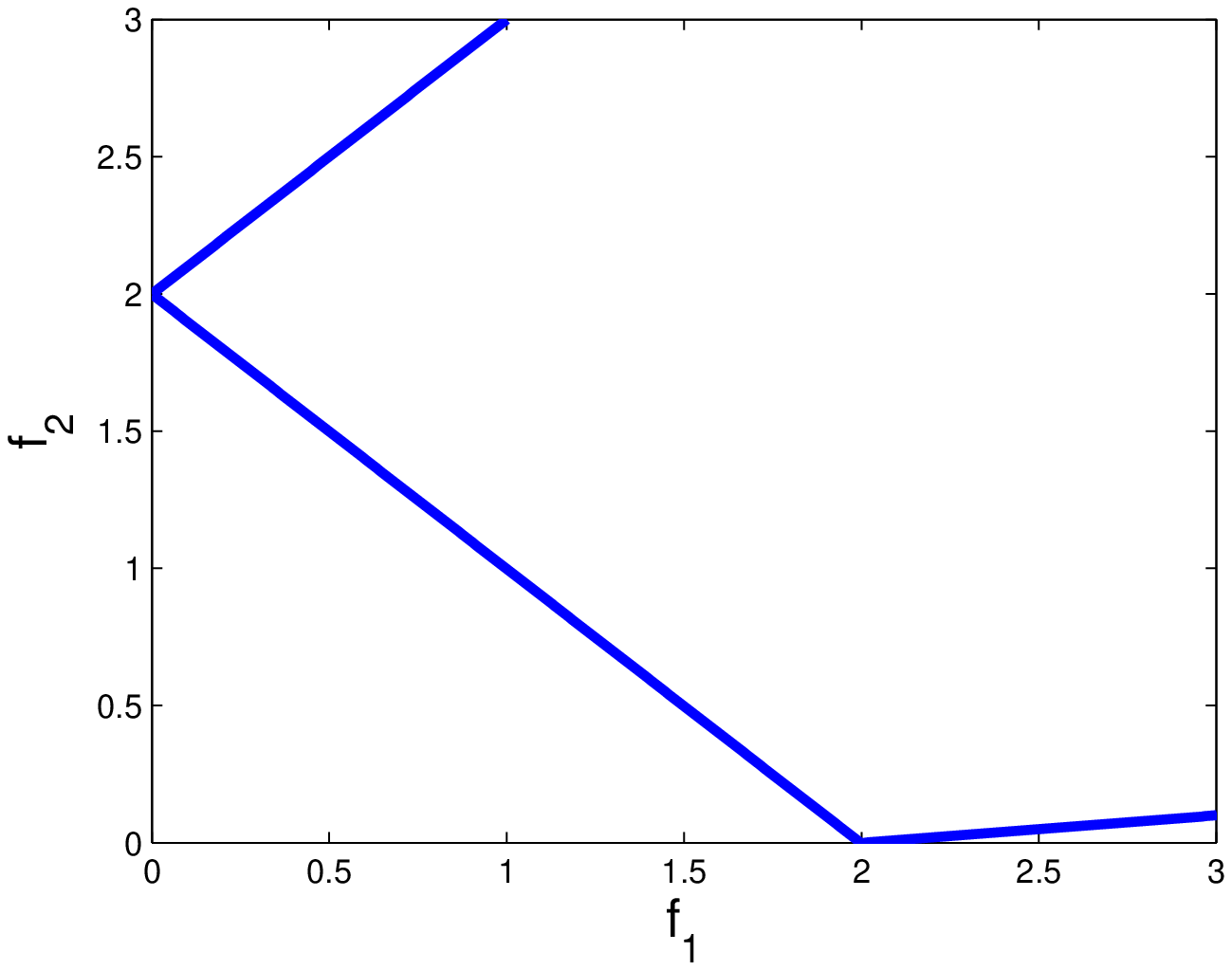,width=6cm}}
\caption{Example of MOP (\ref{eq:exam2}) for $\alpha=0.1$.}
\label{fig:exam2}
\end{center}
\end{figure}
\end{exam}

\section{Bounds on the Archive Sizes}
Here we give an upper bound $U$ on the size of the limit archive obtained by the 
novel strategy, and discuss that the order of $U$ is already optimal.

\begin{thm}
\label{thm:bounds}
Let $\epsilon\in\mathbbm{R}_+^k$, $\Delta^*, \Delta\in\mathbbm{R}_+$ with $\Delta^*<\Delta$ 
be given. Further let $m_i = \min_{x \in Q} f_i(x)$ and 
$M_i = \max_{x \in Q} f_i(x),$ $1 \leq i \leq k$, and $l_0$ as in Theorem \ref{thm:convergence}.
Then, when using $\mathit{ArchiveUpdateP_{Q,\epsilon}}$, the archive size maintained in
Algorithm~\ref{alg:generic_emo} for all $l \geq l_0$ is bounded as

\begin{equation}
\label{eq:bound}
  |A_l| \leq \left(\frac{1}{\Delta^*}\right)^k 
        \sum\limits_{i=1}^k(\epsilon_i+2\Delta+\Delta^*)\prod\limits_{j=1\atop j\neq 
   i}^k(M_j-m_j+\Delta^*).
\end{equation}
\end{thm}

\begin{proof}
Let $l\geq l_0$ and the archive $A_l$ be given. Since $A_l\subset P_{Q,\epsilon+2\Delta}$ 
(see Theorem \ref{thm:convergence}) we are interested in an upper bound on the volume of
$F(P_{Q,\epsilon+2\Delta})$. For this, we consider first the $(k-1)$-dimensional volume
of the Pareto front $F(P_Q)$. Due to the nature of nondominance we can assume that
$F(P_Q)$ is located in the graph of a map
\begin{equation}
\label{eq:dominating_map}
\begin{split}
\Phi_f &:K\to\mathbbm{R}^k\\
\Phi_f(u_1,\ldots,u_{k-1}) &= (u_1,\ldots,u_{k-1},f(u_1,\ldots,u_{k-1})),
\end{split}
\end{equation}
where $K:=[m_1,M_1]\times \ldots \times[m_{k-1},M_{k-1}]$ and $f:K\to [m_k,M_k]$ which
satisfies the monotonicity conditions
\begin{equation}
\label{eq:monotonicity1}
\begin{split}
f(u_1,\ldots,u_{i-1},x_i,u_{i+1},\ldots,u_{k-1}) \leq f(u_1,\ldots,u_{i-1},y_i,u_{i+1},\ldots,u_{k-1})\\
\forall i=1,\ldots, k-1,\\
u_j\in[m_j,M_j], j=1,\ldots,i-1,i+1,\ldots,k-1,\\
x_i, y_i\in [m_i,M_i], x_i\leq y_i.
\end{split}
\end{equation}
Further, we can assume that $f$ is sufficiently smooth. If not, we can replace $f$ by
a smooth function $\tilde{f}$ such that the volume of $\Phi_{\tilde{f}}$ is larger
than the volume of $\Phi_{f}$ as the following discussion shows: \\
if $f(m_i,\ldots,m_{k-1}) = m_k$, the Pareto front consists of one point,
$F(P_Q) = \left\{(m_1,\ldots,m_k)\right\}$, and has minimal volume $0$. Since we are interested
in upper bounds on the volume we can omit this case. Doing so, a smooth function
$\tilde{f}:K\to [m_k,M_k]$ exists with $\tilde{f}(x)\leq f(x),\, \forall x\in K$ that
fulfills the monotonicity conditions (\ref{eq:monotonicity1}). The $(k-1)$-dimensional volume
of $\Phi_{\tilde{f}}$ is obviously larger than the volume of $\Phi_{f}$. Under the smoothness
assumption we can replace condition (\ref{eq:monotonicity1}) by

\begin{equation}
\label{eq:monotonicity2}
  \frac{\partial f}{\partial u_i} u \leq 0,\quad \forall u\in K, \; \forall i=1,\ldots,k-1.
\end{equation}
The $(k-1)$-dimensional volume of $\Phi_f$ with parameter range $K$ is given by 
(see \cite{forster:84}):
\begin{equation}
  Vol_{k-1}(\Phi_f) = \int_K \sqrt{||\nabla f||^2 + 1}du,
\end{equation}
where $\nabla f$ denotes the gradient of $f$.
Analogue to \cite{slcdt:07} (see also Appendix 1) the volume can be estimated by using 
partial integration and the monotonicity conditions (\ref{eq:monotonicity2}) by:
\begin{equation}
\label{eq:estimation_of_volume}
Vol_{k-1}(\Phi_f) \leq \sum\limits_{i=1}^k\prod\limits_{j=1\atop j\neq i}^k(M_j-m_j). 
\end{equation}
Considering this and the nature of $-\epsilon$-dominance we can bound the $k$-dimensional 
volume of $F(P_{Q,\epsilon+2\Delta})$ by: 
\begin{equation}
Vol_k(F(P_{Q,\epsilon+2\Delta})) \leq \sum\limits_{i=1}^k (\epsilon_i+2\Delta) 
 \prod\limits_{j=1\atop j\neq i}^k(M_j-m_j), 
\end{equation}

Since $\|F(a_1)-F(a_2)\|\geq \Delta^*$ for 
all $a_1, a_2\in A_l$ it follows that the boxes
\begin{equation}
   B_{\frac{1}{2}\Delta^*}^\infty(F(a)),\quad a\in A_l,
\end{equation}
are mutually nonoverlapping. Further, if $F(a)\in F(P_{Q,\epsilon+2\Delta})$, then
$B_{\frac{1}{2}\Delta^*}^\infty(F(a))$ is included in a $\Delta^*/2$-neighborhood 
$\tilde{F}$ of $F(P_{Q,\epsilon+2\Delta})$ with 
\begin{equation}
 Vol_k(\tilde{F}) \leq \sum\limits_{i=1}^k (\epsilon_i+2\Delta+\Delta^*) 
 \prod\limits_{j=1\atop j\neq i}^k(M_j-m_j+\Delta^*). 
\end{equation}
The maximal number of entries in $A_l$ can now be estimated by
\begin{equation}
|A_l| \leq \frac{Vol_k(\tilde{F})}{Vol_k(B_{\frac{1}{2}\Delta^*}^\infty(F(a)))},
\end{equation}
and the claim follows since the volume of $B_{\frac{1}{2}\Delta^*}^\infty(F(a))$ is 
obviously given by $(\Delta^*)^k$. 
\end{proof}

In particular interesting is certainly the growth of the magnitudes of the (limit)
archives for vanishing discretization parameter $\Delta$. Since for every meaningful
computation the value $\Delta$ will be smaller than every entry of $\epsilon$, we can 
assume $\epsilon_i = c_i \Delta$ with $c_i>1$. Using (\ref{eq:bound}) and for simlicity
$\Delta=\Delta^*$ we see that
\begin{equation}
  |A_l| \leq \left(\frac{1}{\Delta}\right)^{k-1} 
  \sum\limits_{i=1}^k (c_i+3) \prod\limits_{j=1\atop j\neq i}^k(M_j-m_j+\Delta^*). 
\end{equation}
Thus, the growth of the magnitudes is of order 
$\mathcal{O}\left(\left(\frac{1}{\Delta}\right)^{k-1}\right)$ for $\Delta\to 0$.
Regarding the fact that $P_Q$, which is contained in $P_{Q,\epsilon}$ for all values
of $\epsilon\in\mathbbm{R}_+^k$, typically forms a 
$(k-1)$-dimensional object, we see that the order of the magnitude of the archive with
respect to $\Delta$ is already optimal. This is due to the fact that the discretization 
(line 3 of Algorithm 2) is realized in image space. An analogue result for a discretization 
in parameter space, however, can not hold since $P_{Q,\epsilon}$ is $n$-dimensional.

\begin{rmk}
In case the algorithm is modified as described in Remark \ref{rmks:delta} (c), the upper
bound for the magnitude of the archive is given by
\begin{equation}
|A_l^x| \leq \left(\frac{1}{\Delta^*}+1\right)^n \prod_{j=1}^n (b_i-a_i),
\end{equation}
where $Q\subset [a_1,b_1]\times\ldots\times [a_n,b_n]$ ($P_{Q,\epsilon+2\Delta}$ is certainly
included in $[a_1,b_1]\times\ldots\times [a_n,b_n]$, and maximal $1/\Delta^* + 1$ 
elements can be placed in each coordinate direction). To see that this bound is tight
we consider the example
\begin{equation}
  F:[0,1]^n\to\mathbbm{R}^k,\quad F(x)\equiv c_0\in\mathbbm{R}^k,
\end{equation}
and let $\Delta=1/s,\; s\in\mathbbm{N}$. Define 
$x_{i_1,\ldots,i_n}=(i_1\Delta,\ldots,i_n\Delta)$ and
\begin{equation}
  \mathcal{D} := \left\{ x_{i_1,\ldots,i_n} | 0\leq i_1,\ldots,i_n\leq s \right\}.
\end{equation}
Since $\mathcal{D}\subset [0,1]^n$ and $d_\infty(z_1,z_2)\geq \Delta > \Delta^*$
for all $z_1, z_2\in \mathcal{D}$, $z_1\neq z_2$, all points in $\mathcal{D}$ will
be accepted by the archiver (assuming that only points $z\in \mathcal{D}$ are 
inserted) leading to an archive $A$ with $|A|=|\mathcal{D}|=(s+1)^n$.\\
Since $P_{Q,\epsilon}$ is $n$-dimensional, the growth of the magnitudes of the
archives is also beyond this constructed example of order 
$\mathcal{O}\left(\left(\frac{1}{\Delta}\right)^n\right)$ for $\Delta\to 0$. This
makes a huge difference to the other archiver since for general MOPs we have
$n\gg k$.
\end{rmk}

\section{Numerical Results}
Here we demonstrate the practicability of the novel archiver on five examples. For this,
we run and compare $ArchiveUpdateP_{Q,\epsilon}$ for different values of $\Delta$ including
$\Delta_0 = 0$, which is the archiver which accepts all test points which are not
$-\epsilon$ dominated by any other test point (see Remark 4.3 (a)). To obtain a fair comparison 
we have decided to take a random search operator for the generation process (the same 
sequence of points for all settings). An implementation of the archiver including these 
examples can be found in \cite{url:archives}.

\subsection{Example 1}
First we consider the MOP suggested by Tanaka (\cite{tanaka:95}):
\begin{equation}
\label{eq:MOP_Tanaka}
F :\mathbbm{R}^2 \to \mathbbm{R}^2, \quad F(x_1,x_2) = (x_1, x_2)
\end{equation}
\indent where
\begin{equation*}
\begin{split}
  C_1(x) &= x_1^2 + x_2^2 - 1 - 0.1\cos(16 \arctan (x_1/x_2)) \geq  0 \\
  C_2(x) &= (x_1-0.5)^2 + (x_2-0.5)^2 \leq 0.5
\end{split}
\end{equation*}

Figure \ref{fig:tanaka} shows a numerical result for $N=200,000$ randomly chosen
points within $Q=[0,\pi]^2$ and for three different values of the discretization parameter:
$\Delta_0=0$, $\Delta_1=0.01$ and $\Delta_2=0.05$. As anticipated, the granularity
of the resulting archive increases with the value of $\Delta$. Thus, the approximation
quality decreases, but also the running time of the algorithm (see Table \ref{tab:tanaka}).

\begin{figure}
\begin{center}
\subfigure[$\Delta_0=0,\quad |A_{final}|=3824$]{
\epsfig{file=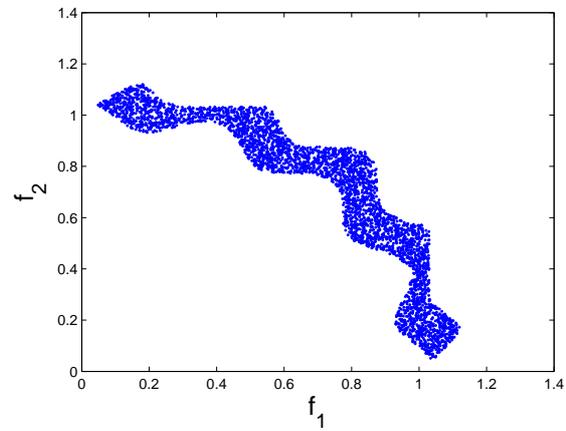,width=8cm}}
\vfill
\subfigure[$Delta_1=0.01,\quad |A_{final}|=834$]{
\epsfig{file=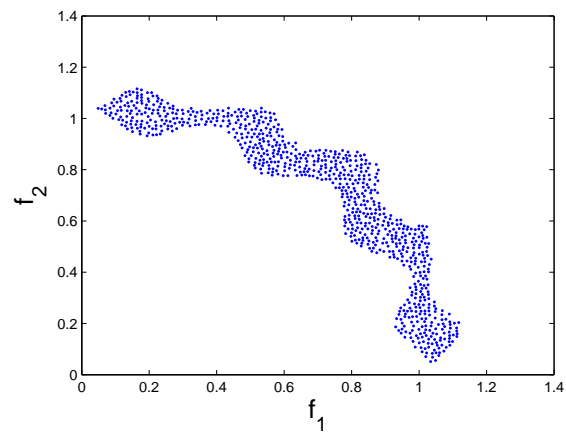,width=8cm}}
\vfill
\subfigure[$Delta_2=0.05,\quad |A_{final}|=73$]{
\epsfig{file=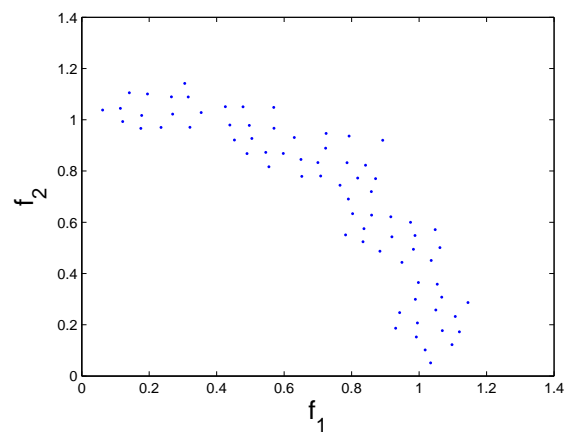,width=8cm}}
\caption{Results for MOP (\ref{eq:MOP_Tanaka}) for different values of 
$\Delta$ leading to different granularities of the approximation.}
\label{fig:tanaka}
\end{center}
\end{figure}

\begin{table}
\begin{center}
\caption{Comparison of the magnitudes of the final archive ($|A_{final}|$, rounded) and the 
corresponding update times ($T$, in seconds) for MOP (\ref{eq:MOP_Tanaka}) and for different
values of $\Delta$. We have taken the average result of 100 test runs.}
\label{tab:tanaka}
\begin{tabular}{c||c|c}
$\Delta$ & $|A_{final}|$ & $T$\\\hline\hline
$0$& 3836 & 32.98 \\
$0.01$ & 827 & 6.22\\
$0.05$ & 68 & 1.80
\end{tabular}
\end{center}
\end{table}

\subsection{Example 2}
Next, we consider the following MOP proposed in \cite{rudolph:06}:
\begin{equation}
\label{eq:MOP_Rudolph}
\begin{split}
F &:\mathbbm{R}^2 \to \mathbbm{R}^2\\
F(x_1,x_2) &= \left(\begin{tabular}{c} $(x_1 -t_1(c+2a)+a)^2 + (x_2-t_2b)^2$\\
   $(x_1 -t_1(c+2a)-a)^2 + (x_2-t_2b)^2$ \end{tabular}\right),
\end{split}
\end{equation}
\indent where
\[
t_1 = \mbox{sgn}(x_1) \min\left(\left\lceil \frac{|x_1|-a-c/2}{2a+c}\right\rceil, 1\right),
t_2 = \mbox{sgn}(x_2)\min\left(\left\lceil \frac{|x_2|-b/2}{b}\right\rceil, 1\right).
\]
The Pareto set consists of nine connected components of length $a$ with identical images. 
We have chosen the values $a=0.5$, $b=c=5$, $\epsilon=(0.1,0.1)$, and the domain as 
$Q=[-20,20]^2$. Figure \ref{fig:rudolph} shows some numerical results for $N=100,000$
randomly chosen points within $Q$ and for the three variants of the archiver
$ArchiveUpdateP_{Q,\epsilon}$: (a) $\Delta = (0,0)$, i.e., the archiver which aims to store
the entire set $P_{Q,\epsilon}$, (b) a discretization (in image space) using
$\Delta = (0.02, 0.02)$, and (c) the variant which is described in 
Remark \ref{rmks:delta} (c) using $\Delta_x=(0.1,0.1)$ for a discretization of the 
parameter space. As anticipated, the solution in (b) is more uniform in image space
compared to the solution in (c), which is, in turn, more uniform in parameter space. 
Note that the solution in (b), i.e., where the discretization has been done in image
space, already 'detects' all nine connected components in parameter space. This is
certainly due to the fact that the search has been done by selecting random points
wich are uniformly distributed within $Q$. One interesting point for future studies
would be to investigate if this capability still holds when other search strategies
are chosen (which include, e.g., local search strategies). 

\begin{figure}
\begin{center}
\subfigure[$\Delta = (0, 0)$, $|A_{final}| = 2150$]{
\epsfig{file=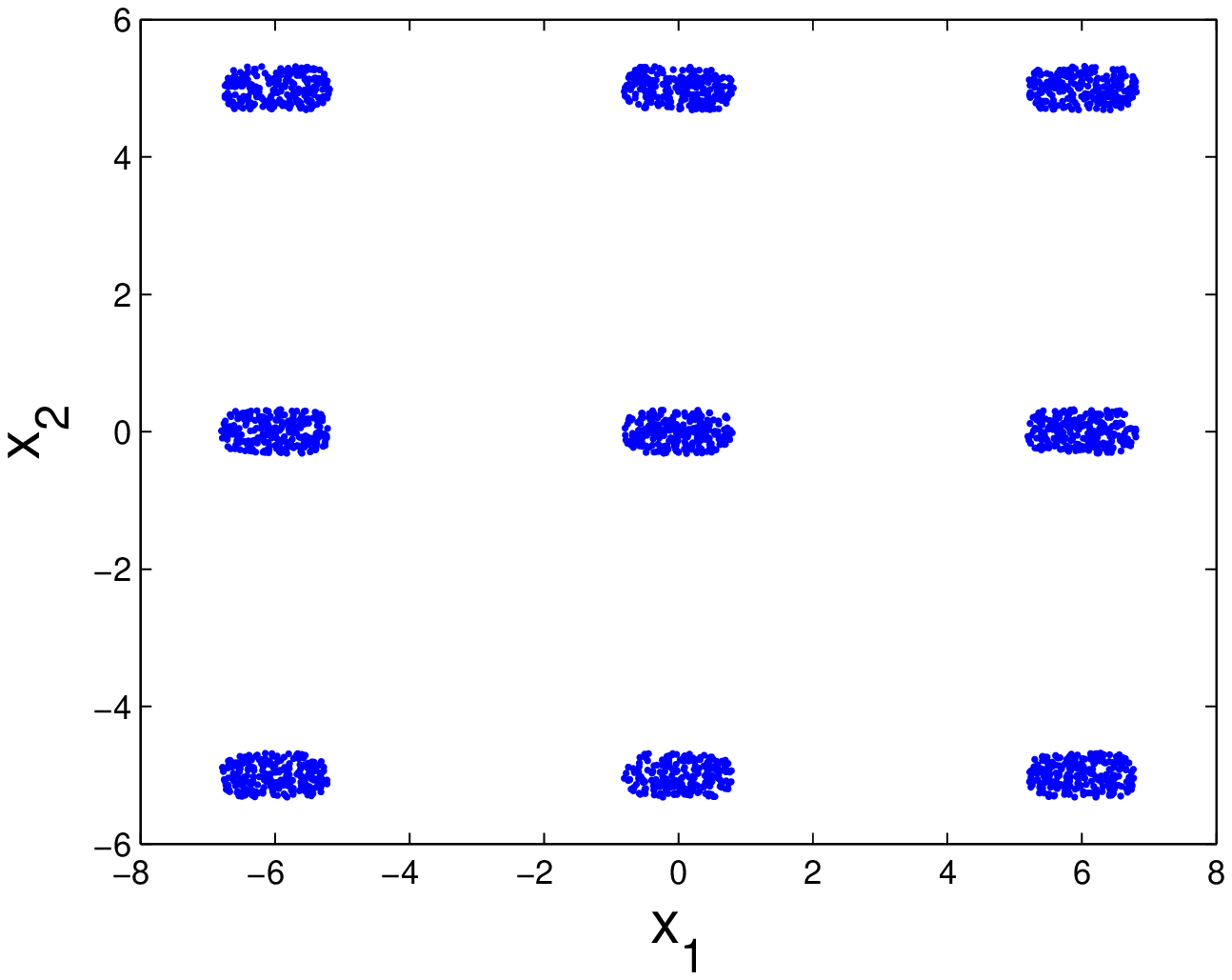,width=7cm}
\hfill
\epsfig{file=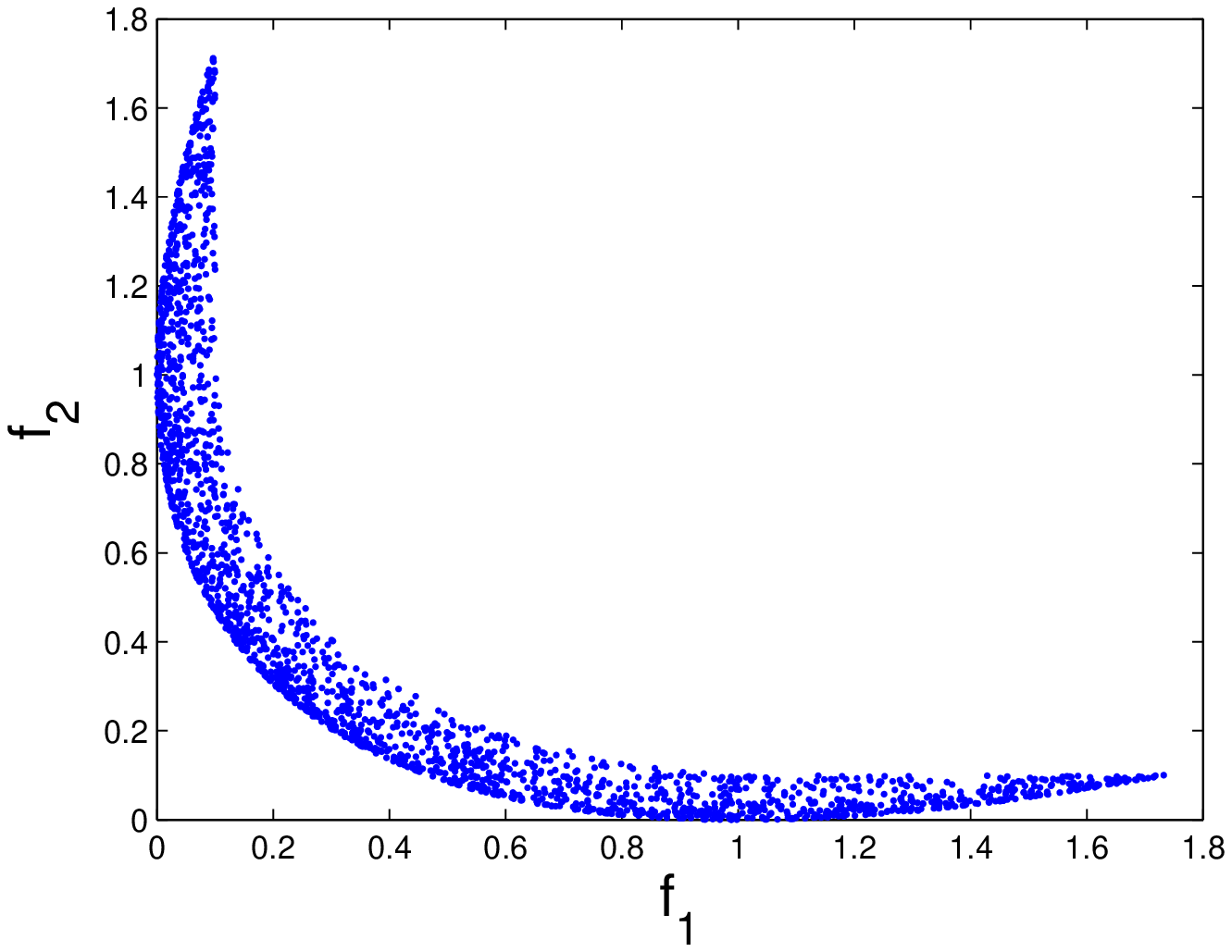,width=7cm}}
\vfill
\subfigure[$\Delta = (0.02, 0.02)$, $|A_{final}| = 365$ ]{
\epsfig{file=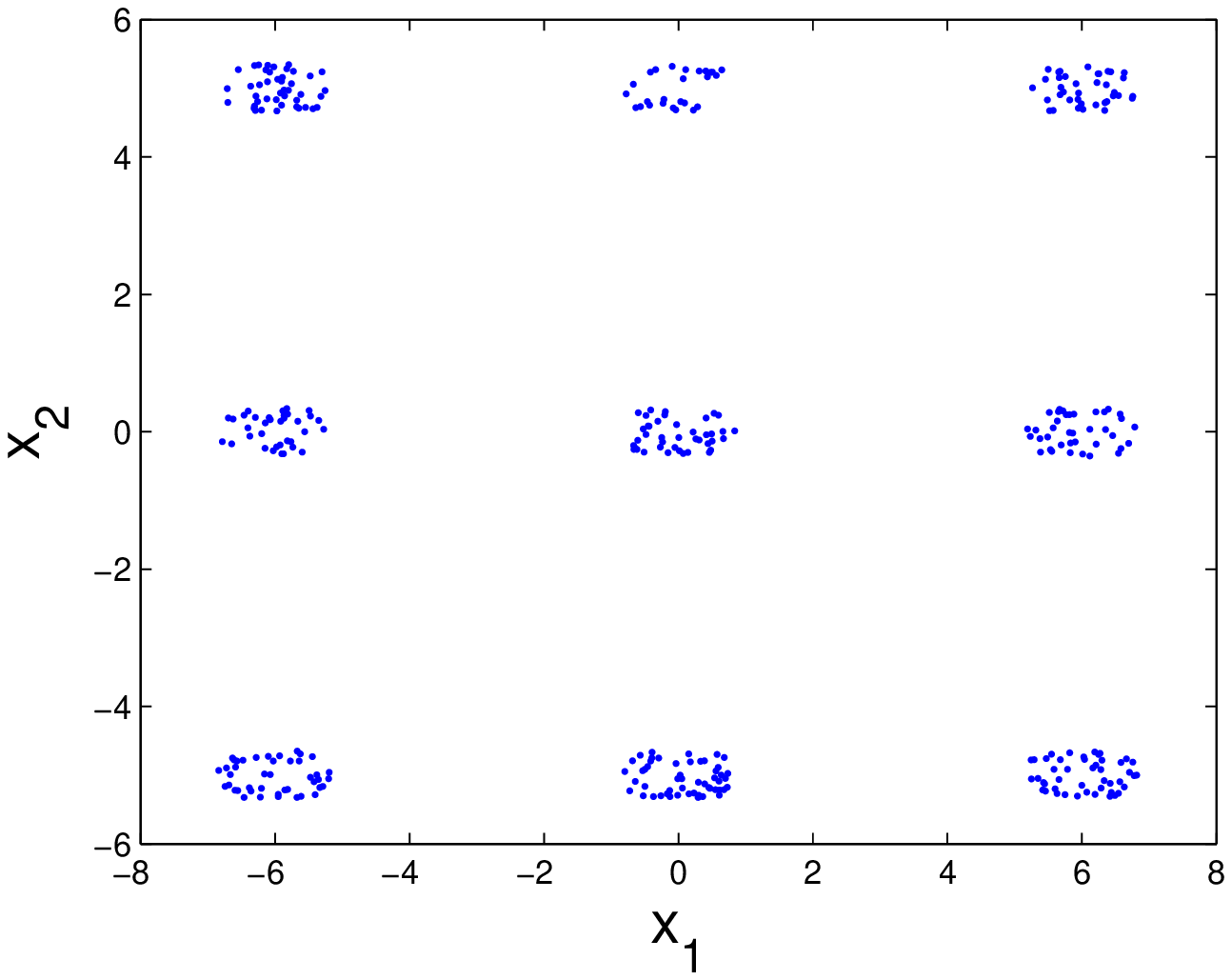,width=7cm}
\hfill
\epsfig{file=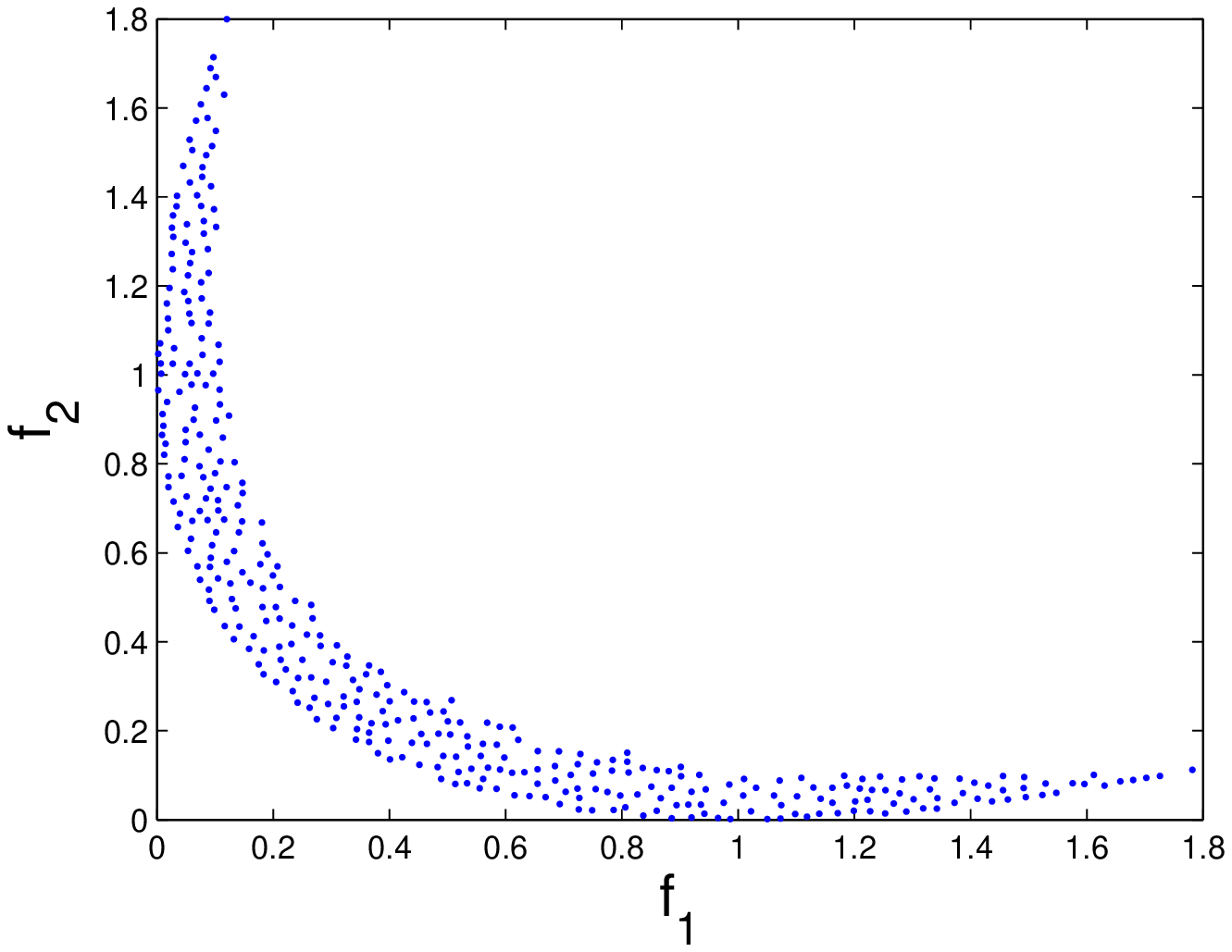,width=7cm}}
\vfill
\subfigure[$\Delta_x = (0.1, 0.1)$, $|A_{final}| = 410$]{
\epsfig{file=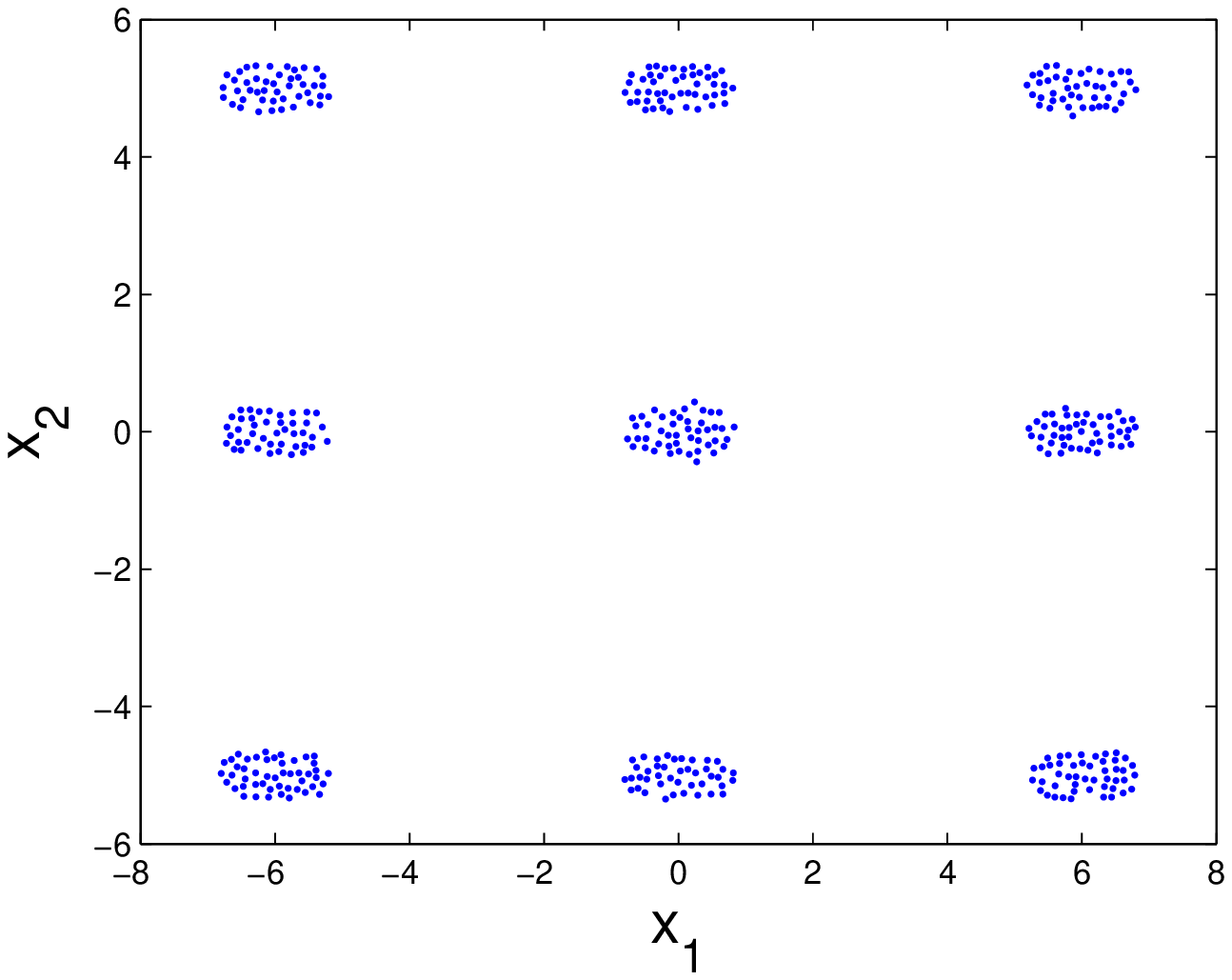,width=7cm}
\hfill
\epsfig{file=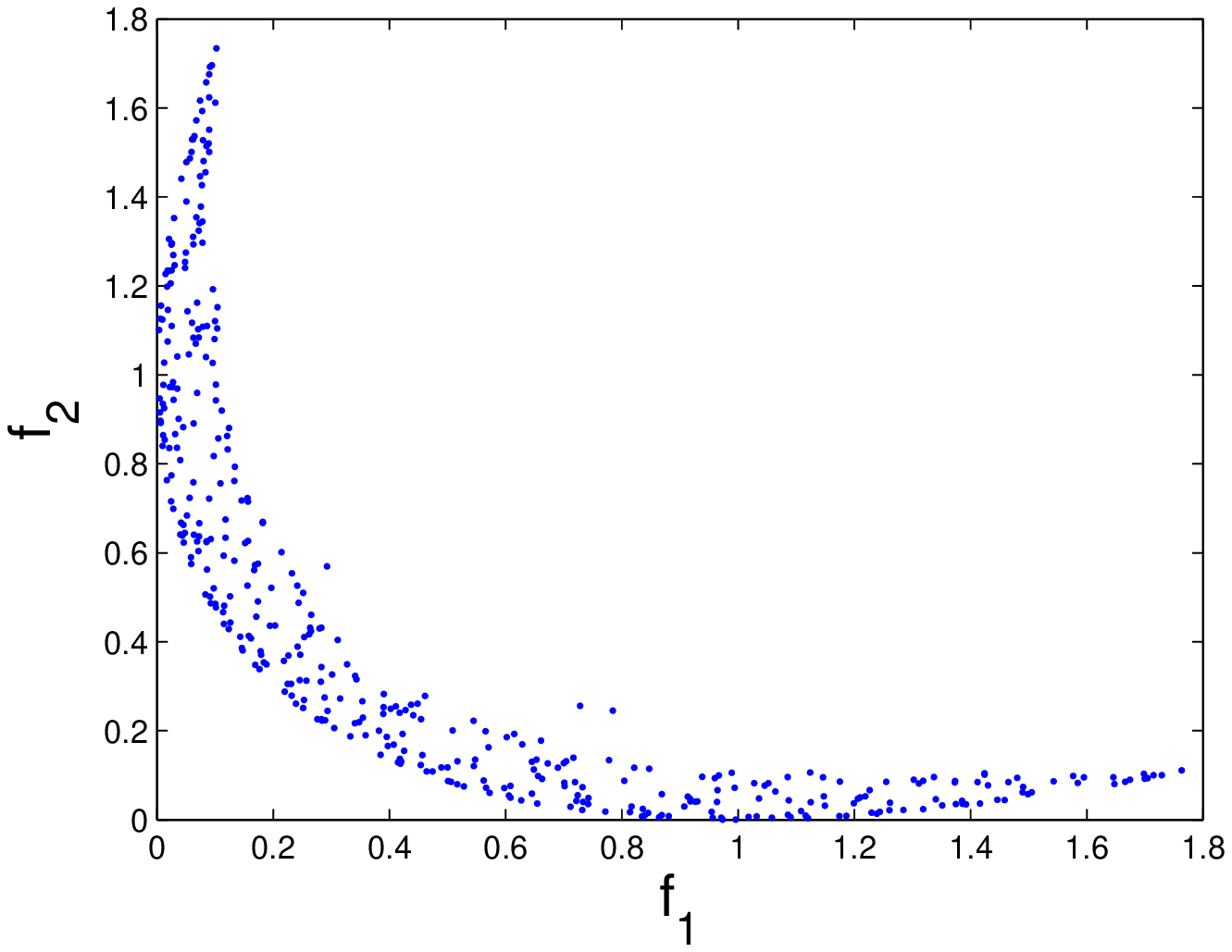,width=7cm}}
\caption{Numerical results for MOP (\ref{eq:MOP_Rudolph}) for the three
different variants of $ArchiveUpdateP_{Q,\epsilon}$.}
\label{fig:rudolph}
\end{center}
\end{figure}

\subsection{Example 3}
Finally, we consider the production model proposed in \cite{SSW:02}:

$f_1,f_2:\mathbbm{R}^{n}\to\mathbbm{R}$,
\begin{equation}
\label{eq:MOP_SSW}
\begin{split}
f_1(x) &= \sum_{j=1}^{n}x_j, \\
f_2(x) &= 1 - \prod_{j=1}^n(1-w_j(x_j)), 
\end{split}
\end{equation}
where
\begin{eqnarray*}
&w_j(z) = \left\{\begin{array}{l@{\quad}l}
 0.01\cdot \exp(-(\frac{z}{20})^{2.5}) & \quad\mbox{for}\quad j=1,2 \\[2mm]
 0.01\cdot \exp(-\frac{z}{15}) &  \quad\mbox{for}\quad 3\leq j \leq n
\end{array}\right.
\end{eqnarray*}

The two objective functions have to be interpreted as follows. $f_1$ represents the sum 
of the additional cost necessary for a more reliable production of $n$ items.  These items 
are needed for the composition of a certain product. The function $f_2$ describes the total 
failure rate for the production of this composed product.\\
Here we have chosen $n=5$, $Q=[0,40]^n$, and $\epsilon = (0.1, 0.001)$ which corresponds
to $10$ percent of one cost unit for one item ($\epsilon_1$), and to $0.1$ percent of the 
total failure rate ($\epsilon_2$). Figure \ref{fig:SSW} shows numerical results for
(a) $\Delta= (0,0)$ and (b) for $\Delta = \epsilon/3$. 
Note the symmetries in the model: it is e.g., $F(x_1,x_2,\ldots) = F(x_2,x_1,\ldots)$ by
which is follows that the two connected components at $x_1=0$ and $x_2=0$ (see
Figure \ref{fig:SSW} (a)) have the same image. Also in this case the archiver detects
both components though the discretization has been done in image space 
(Figure \ref{fig:SSW} (b)).

\begin{figure}
\begin{center}
\subfigure[$\Delta = (0, 0)$, $|A_{final}| = 5939$]{
\epsfig{file=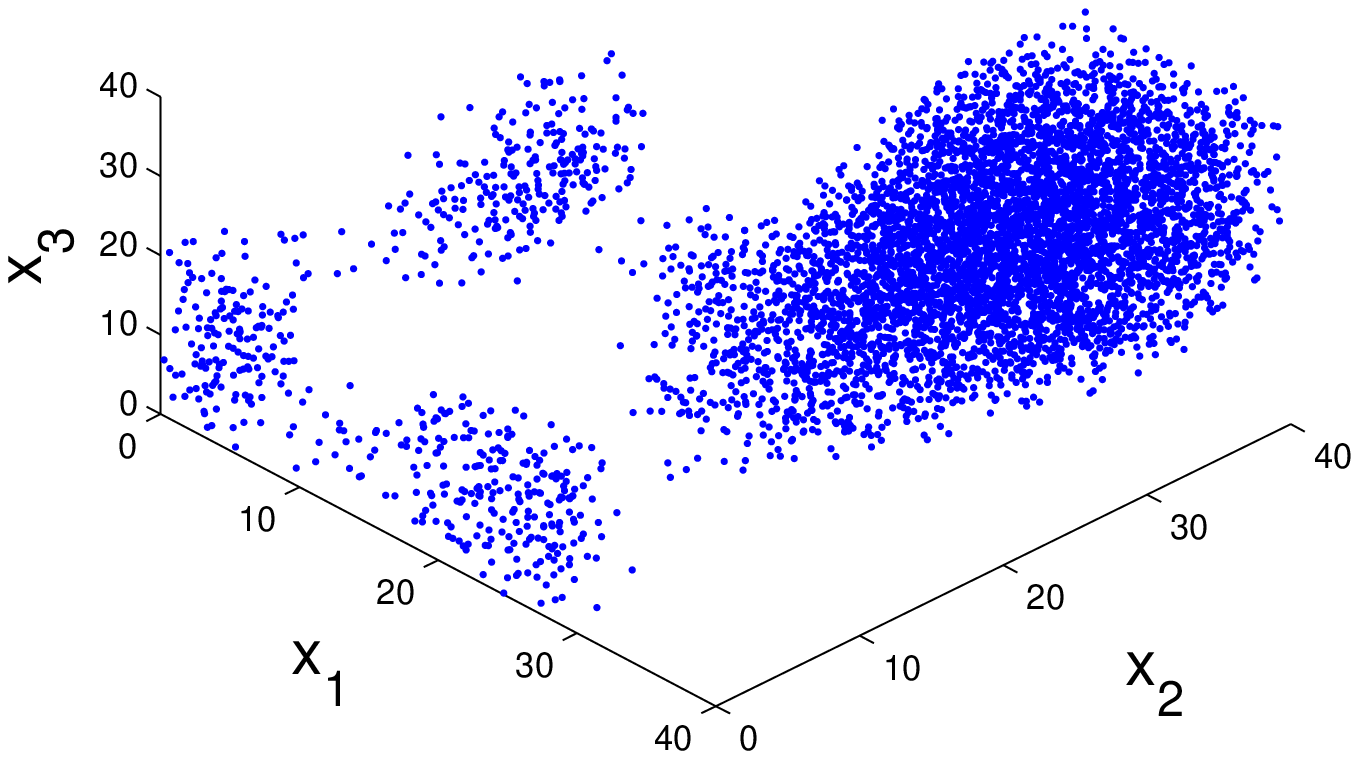,width=7cm}
\hfill
\epsfig{file=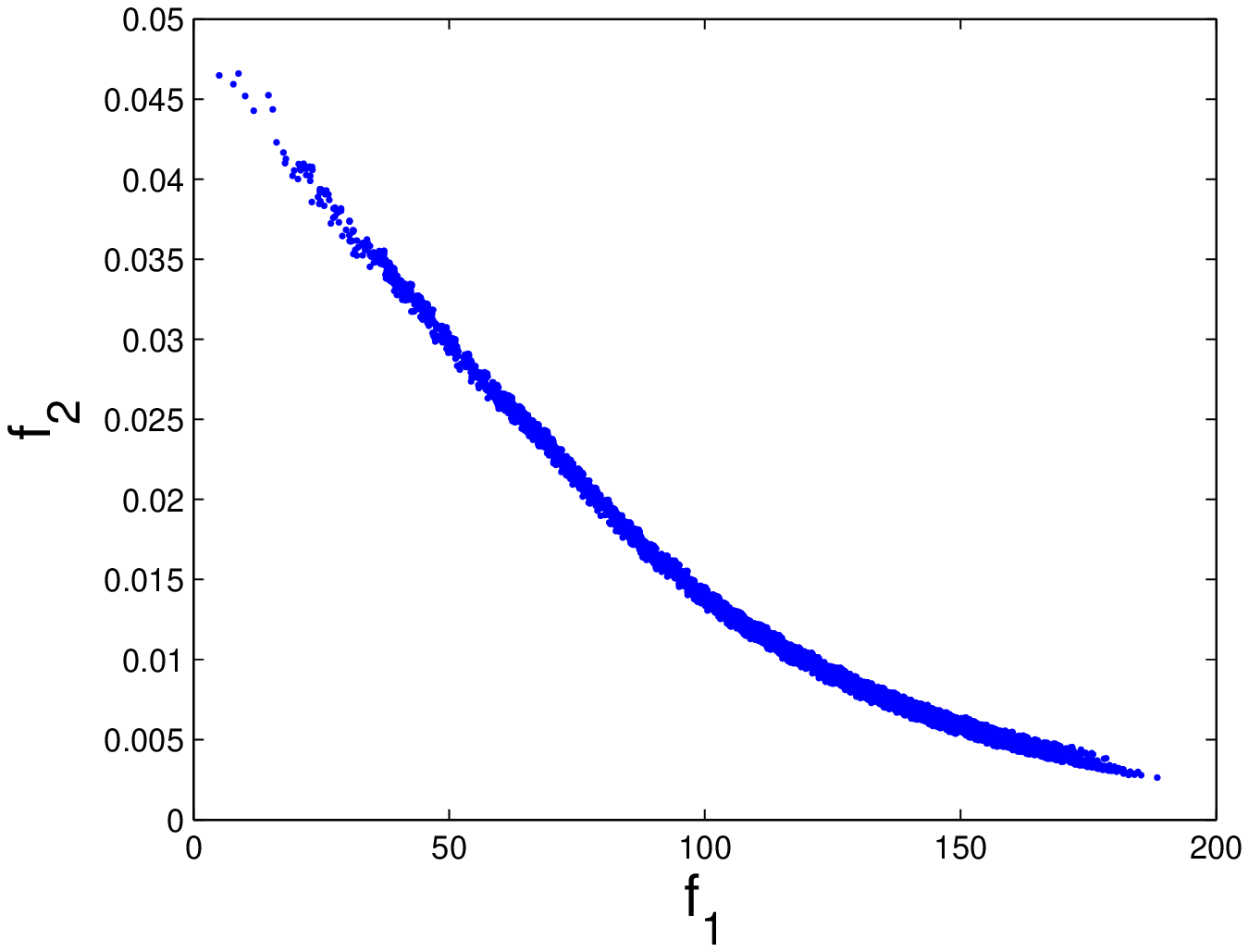,width=7cm}}
\vfill
\subfigure[$\Delta = (0.033, 0.00033)$, $|A_{final}| = 3544$]{
\epsfig{file=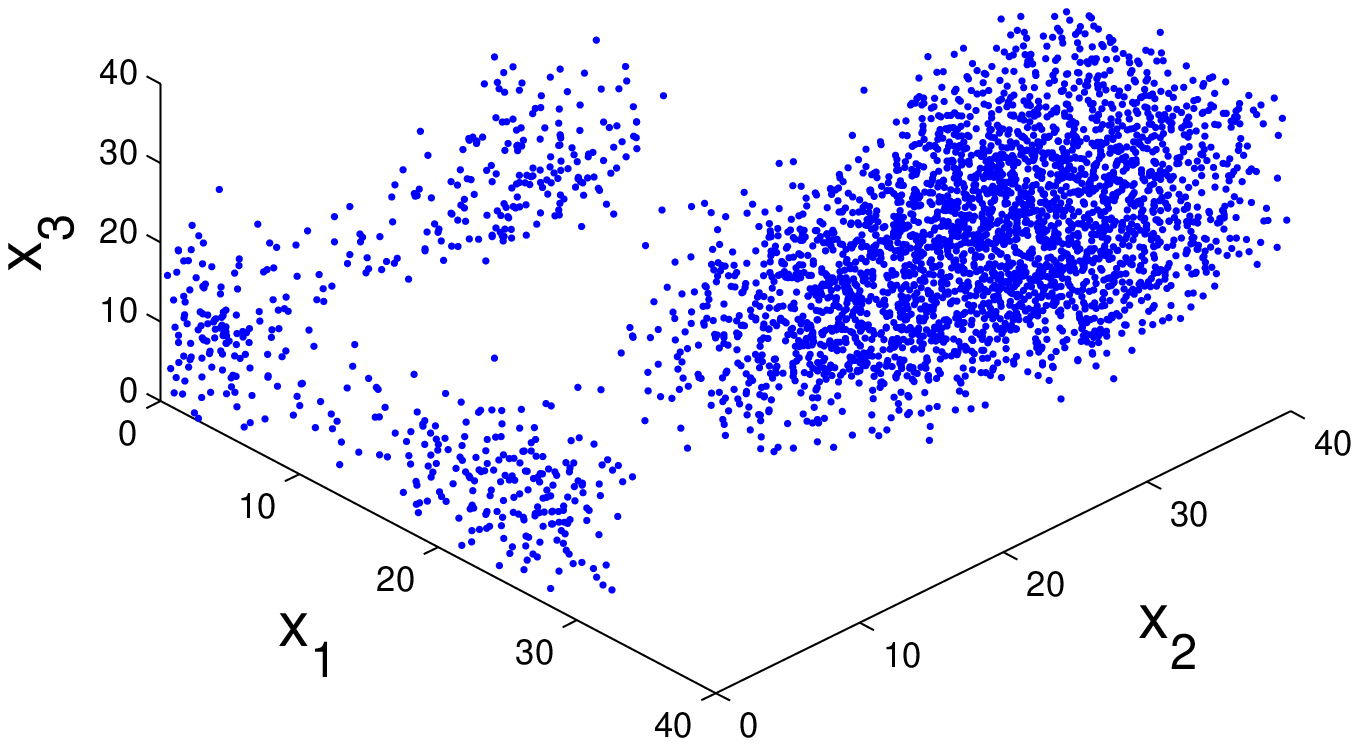,width=7cm}
\hfill
\epsfig{file=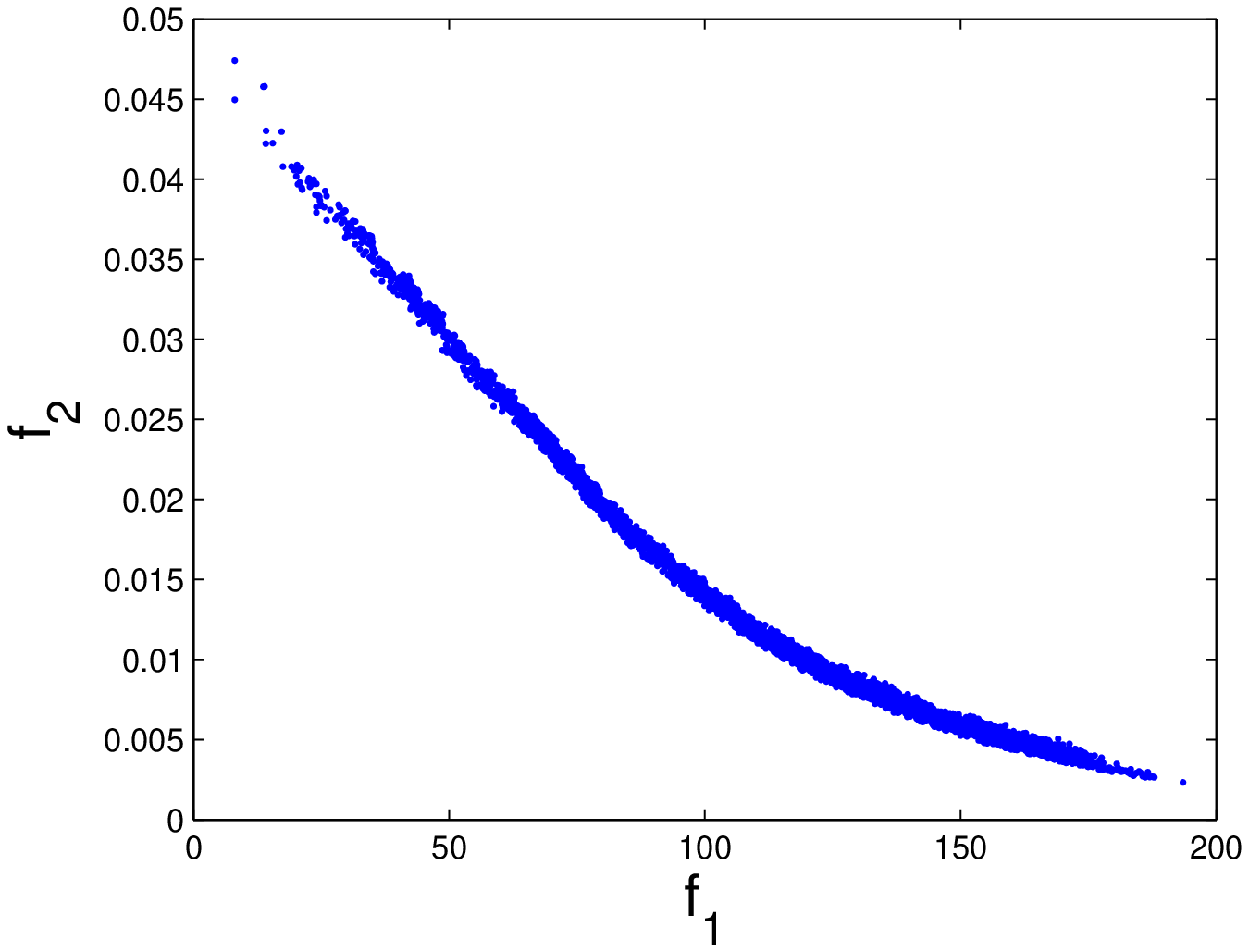,width=7cm}}
\caption{Numerical result for MOP (\ref{eq:MOP_SSW}): projections to the coordinates
$x_1, x_2, x_3$ of the final archive (left) and their images (right).}
\label{fig:SSW}
\end{center}
\end{figure}

\subsection{Example 4}
Next, we consider a real-life engineering problem, namely the design
of a four-bar plane truss (\cite{stadler:92}): 
\begin{equation}
\label{eq:MOP_truss}
\begin{split}
F &:\mathbbm{R}^4 \to \mathbbm{R}^2\\
f_1(x) &= L(2x_1+\sqrt{2}x_2 + \sqrt{2}x_3 + x_4)\\
f_2(x) &= \frac{FL}{E}\left( \frac{2}{x_1} + \frac{2\sqrt{2}}{x_2}
           - \frac{2\sqrt{2}}{x_3} + \frac{1}{x_4}\right)
\end{split}
\end{equation}
$f_1$ models the volume of the truss, and $f_2$ the displacement of the joint. The model 
constants are the length $L$ of each bar ($L=200$ cm), the elasticity constants $E$ and $\sigma$
of the materials involved ($E = 2\times 10^5$ kN/cm$^3$, $\sigma=10$ kN/cm$^2$), and
the force $F$ which causes the stress of the truss ($F=10$ kN). The parameters $x_i$ 
represent the cross sectional areas of the four bars of the truss. The physical restrictions 
are given by
\begin{equation}
 Q = [F/\sigma,3F/\sigma]\times[\sqrt{2}F/\sigma,3F/\sigma]^2\times[F/\sigma,3F/\sigma]
\end{equation}
For the allowed tolerances we follow the suggestion made in \cite{engau:07} and set 
$\epsilon_1 = 50$ cm$^3$ and $\epsilon_2 = 0.0005$ cm. Figure \ref{fig:truss} shows
a result for $N=500,000$ randomly chosen points within $Q$ and for 
$\Delta = (10,0.0001)$, i.e., $\Delta_i = \epsilon_i / 5$ (see Remark 4.3 (b)). The final 
archive consists of 78 elements, and the computational time was 5.5 seconds. In contrast, 
a run of the same algorithm with the same setting but with $\Delta=0$ took 4 minutes and 21 
seconds leading to an archive with 8377 elements.

\begin{figure}
\begin{center}
\epsfig{file=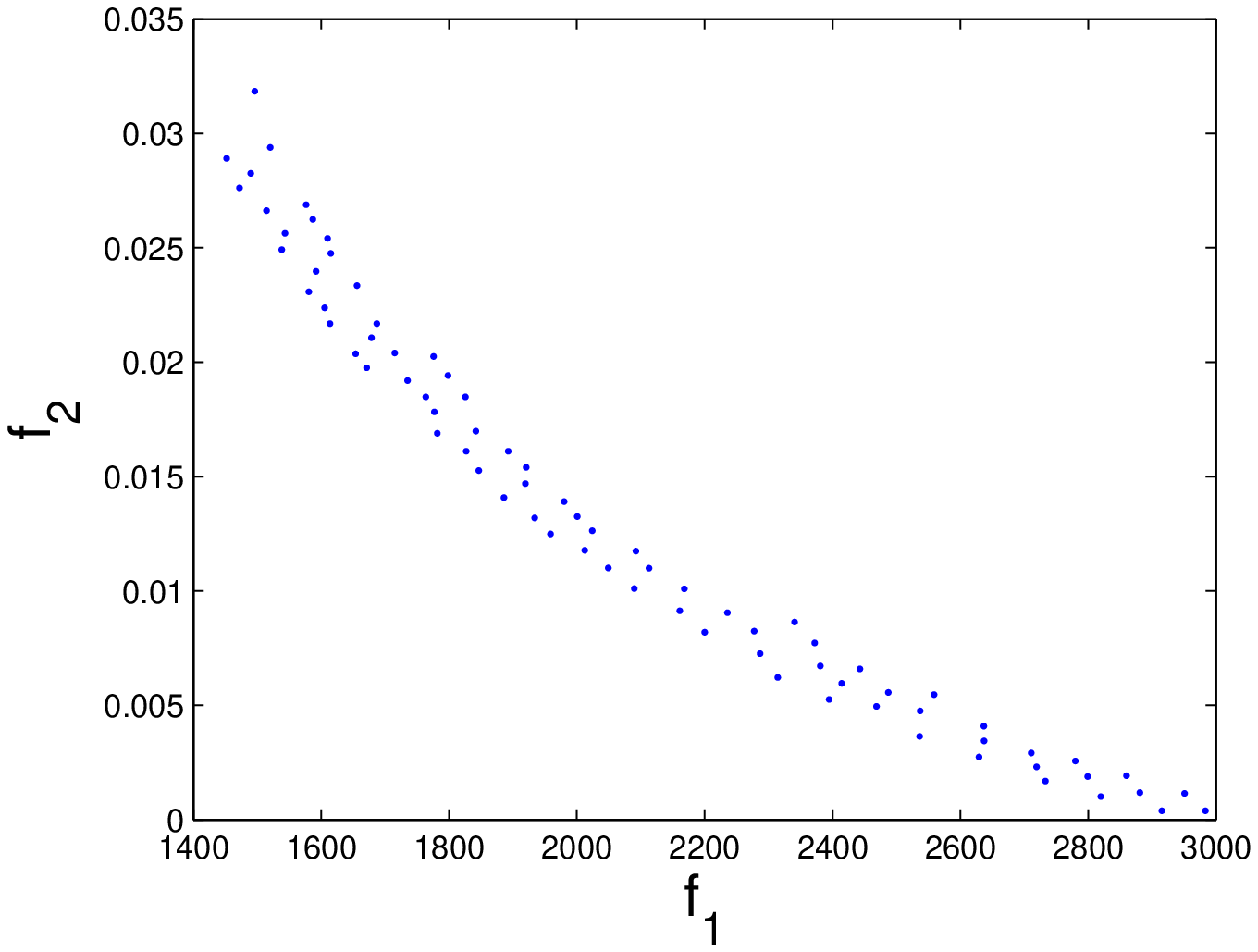,width=10cm}
\caption{Numerical result for MOP (\ref{eq:MOP_truss}). Here, we have chosen
$\epsilon=(50,0.0005)$ and $\Delta = (10,0.0001)$.}
\label{fig:truss}
\end{center}
\end{figure}

\subsection{Example 5}
Finally, we consider a bi-objective \{0,1\}-knapsack problem which should demonstrate
that the additional consideration of approximate solutions can be beneficial for
the decision maker. 
\begin{equation}
\label{eq:knapsack}
 f_1, f_2  : \{0,1\}^n\to\mathbbm{R},\  f_1(x) = \sum\limits_{j=1}^n c_j^1 x_j, \ f_2(x) = \sum\limits_{j=1}^n c_j^2 x_j
\end{equation}
\indent s.t.
\begin{equation*}
   \sum\limits_{j=1}^n w_j x_j \leq W, \ x_j \in \{0,1\},\; j=1,\ldots, n,
\end{equation*}
where $c_j^i$ represents the value of item $j$ on criterion $i, i=1,2$; $x_j=1$, 
$j=1,\ldots,n$, if item $j$ is included in the knapsack, else $x_j=0$. $w_j$ is the 
weight of item $j$, and $W$ the overall knapsack capacity. 
Figure \ref{fig:knp} shows one numerical result obtained by an evolutionary 
strategy\footnote{A modification of the algorithm presented in \cite{tantar:07} which
uses the novel archiver.} for an instance with $n=30$ items and with randomly
chosen values $c_j^i\in [8,12]$, and capacity $W=15$ (note that we are faced with a 
maximization problem). For $\epsilon=(2,2)$ and $\Delta=0.1$ a total of 182 elements 
forms the final archive, and only six of them are nondominated. When taking, for instance,
$x_0$ as reference (assuming, e.g., that this point has been selected by the
DM out of the nondominated points) and assuming a tolerance of 1 which represents a
possible loss of $0.6 \%$ compared to $x_0$ for each objective value, the resulting 
region of interest
includes seven approximate solutions (see Figure \ref{fig:knp}). These solutions, though 
similar in objective space, differ significantly in parameter space: two solutions 
differ compared to $x_0$ in 8 items, one in 10, and 4 solutions differ even in 12 items. 
Thus, in this case it is obvious that by tolerating approximate solutions -- where the loss 
of them can be determined a priori -- a larger variety of possibilities is offered to the DM.

\begin{figure}
\begin{center}
\epsfig{file=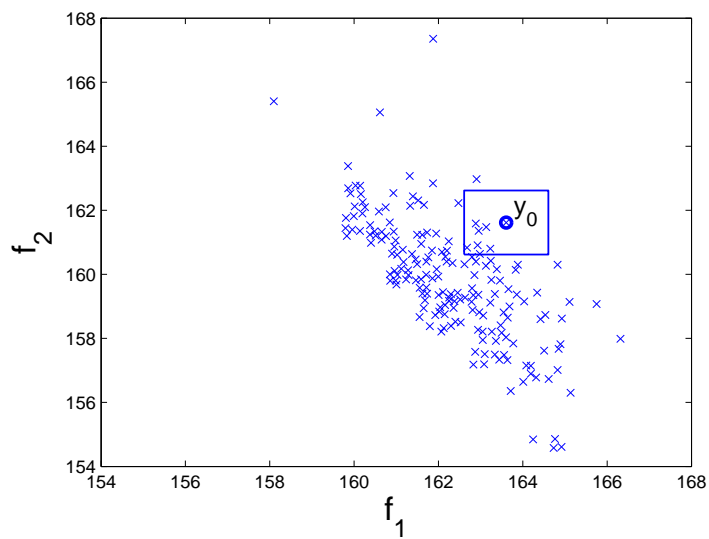,width=10cm}
\caption{Numerical result for MOP (\ref{eq:MOP_truss}) with $\epsilon = (2,2)$ and
$\Delta = 0.1$. The rectangle defines one possible region of interest around 
$y_0 = F(x_0)$ including seven approximate solutions (see text).}
\label{fig:knp}
\end{center}
\end{figure}

\begin{table}
\begin{center}
\caption{..}
\begin{tabular}{|c c||cccccccc|}\hline
$c^1$ & $c^2$ & $x_0$ & $x_1$ & $x_2$ & $x_3$ & $x_4$ & $x_5$ & $x_6$ & $x_7$ \\\hline\hline 
 & & 1 &   1 &   1  &  1 &   1  &  1 &   0 &   0\\
 & & 0 &   0 &   1  &  1 &   0  &  1 &   1 &   0\\
 & & 1 &   1 &   0  &  1 &   0  &  1 &   1 &   1\\
 & & 1 &   1 &   1  &  1 &   1  &  1 &   0 &   1\\
 & & 0 &   1 &   0  &  1 &   1  &  1 &   1 &   1\\
 & & 0 &   0 &   1  &  1 &   0  &  0 &   0 &   0\\
 & & 0 &   1 &   1  &  0 &   1  &  0 &   1 &   1\\
 & & 1 &   0 &   0  &  0 &   1  &  0 &   0 &   0\\
 & & 1 &   1 &   1  &  1 &   0  &  1 &   1 &   1\\
 & & 0 &   0 &   1  &  0 &   0  &  1 &   1 &   0\\
 & & 1 &   0 &   0  &  1 &   1  &  1 &   1 &   1\\
 & & 0 &   0 &   0  &  0 &   0  &  0 &   0 &   0\\
 & & 1 &   1 &   1  &  1 &   1  &  0 &   0 &   0\\
 & & 0 &   0 &   0  &  0 &   0  &  0 &   1 &   0\\
 & & 1 &   1 &   0  &  1 &   1  &  0 &   1 &   0\\
 & & 1 &   1 &   1  &  1 &   1  &  1 &   1 &   1\\
 & & 0 &   1 &   1  &  0 &   0  &  1 &   0 &   1\\
 & & 0 &   0 &   0  &  0 &   0  &  0 &   0 &   0\\
 & & 1 &   1 &   0  &  1 &   0  &  0 &   1 &   1\\
 & & 0 &   0 &   0  &  0 &   0  &  1 &   0 &   1\\
 & & 0 &   0 &   0  &  0 &   1  &  0 &   0 &   1\\
 & & 1 &   1 &   1  &  0 &   1  &  1 &   1 &   1\\
 & & 1 &   0 &   1  &  0 &   1  &  0 &   0 &   0\\
 & & 0 &   1 &   1  &  1 &   1  &  1 &   1 &   1\\
 & & 1 &   0 &   1  &  0 &   0  &  1 &   1 &   0\\
 & & 0 &   0 &   0  &  0 &   0  &  0 &   0 &   0\\
 & & 0 &   1 &   0  &  0 &   0  &  0 &   0 &   0\\
 & & 1 &   1 &   0  &  1 &   1  &  1 &   0 &   1\\
 & & 0 &   0 &   0  &  0 &   0  &  0 &   0 &   0\\
 & & 1 &   0 &   1  &  1 &   1  &  0 &   1 &   1 \\\hline
\end{tabular}
\end{center}
\end{table}

\section{Conclusion and Future Work}
We have proposed and investigated a novel archiving strategy for stochastic search
algorithms which allows -- under mild assumptions on the generation process -- for a
finite size approximation of the set $P_{Q,\epsilon}$ which contains all 
$\epsilon$-efficient solutions of an MOP within a compact domain $Q$. 
We have proven convergence of the algorithm toward a finite size representation of
the set of interest in the probabilistic sense, yielding bounds on the approximation 
quality and the cardinality of the archives. Finally, we have presented some numerical 
results indicating the usefulness of the approach. \\
The consideration of approximate solutions certainly leads to a larger variety of
possible options to the DM, but, in turn, also to a higher demand on the related 
decision making process. Thus, the support for this problem could be one focus of future 
research. Further, it could be interesting to integrate the archiving strategy directly 
into the stoachastic search process  (as e.g. done in \cite{deb:05} for an EMO algorithm) 
in order to obtain a fast and reliable search procedure.

\section*{Acknoledgements}
The second author gratefully acknowledges support from CO\-NA\-CyT 
project no. 45683-Y.
\bibliographystyle{plain}
\bibliography{conv_epseff2}

\clearpage
\section{Appendix 1}
Here we show the inequality in (\ref{eq:estimation_of_volume}), which is taken
from \cite{slcdt:07}. \\
Define
\begin{equation}
\begin{split}
 K&:=[m_1,M_1]\times \ldots \times[m_{k-1},M_{k-1}],\\
 K_{(i)}&:= [m_1,M_1]\times\ldots \times[m_{i-1},M_{i-1}]\times[m_{i+1},M_{i+1}]
         \times\ldots\times[m_{k-1},M_{k-1}],\\
 u_{(i)}&:= (u_1, \ldots, u_{i-1}, u_{i+1}, \ldots,u_{k-1}),\; i=1,\ldots,k-1.
\end{split}
\end{equation}

Then, the $(k-1)$-dimensional volume of $\Phi_f$ with parameter range $K$ can be 
bounded as follows:

\begin{eqnarray}
\begin{split}
Vol_{k-1}(\Phi_f) &= \int_K \sqrt{||\nabla f||^2 + 1}du =
  \int_K \sqrt{\left(\frac{\partial f}{\partial u_1}\right)^2 + \ldots
                 + \left(\frac{\partial f}{\partial u_{k-1}}\right)^2  + 1}du\\
 &\leq \int_K \left|\frac{\partial f}{\partial u_1} \right|du+ \ldots
  +  \int_K \left|\frac{\partial f}{\partial u_{k-1}} \right|du
  + \int_K 1 du\\
 &= \sum_{i=1}^{k-1}\left( \int_{K_{(i)}}
   \left(\int_{m_i}^{M_i}\left|\frac{\partial f}{\partial u_i}\right|du_i\right)
   du_{(i)}\right) + \int_K 1 du\\
 &= \sum_{i=1}^{k-1}\left(\int_{K_{(i)}}
   \left(- \int_{m_i}^{M_i} \frac{\partial f}{\partial u_i} du_i\right)
   du_{(i)}\right) + \int_K 1 du\\
 &\leq \sum\limits_{i=1}^k\prod\limits_{j=1\atop j\neq i}^k(M_j-m_j).
\end{split}
\end{eqnarray}

\end{document}